\documentclass[english]{article}
\usepackage[T1]{fontenc}
\usepackage[latin9]{inputenc}
\usepackage{geometry}
\usepackage{verbatim}
\usepackage{textcomp}
\usepackage{mathtools}
\usepackage{amsmath}
\usepackage{amsthm}
\usepackage{amssymb}
\usepackage{graphicx}
\usepackage{xcolor}
\PassOptionsToPackage{normalem}{ulem}
\usepackage{ulem}

\makeatletter

\theoremstyle{plain}
\newtheorem{thm}{\protect\theoremname}[section]
\theoremstyle{plain}
\newtheorem{assumption}[thm]{\protect\assumptionname}
\theoremstyle{definition}
\newtheorem{defn}[thm]{\protect\definitionname}
\theoremstyle{plain}
\newtheorem{prop}[thm]{\protect\propositionname}
\theoremstyle{definition}
\newtheorem{example}[thm]{\protect\examplename}
\theoremstyle{plain}
\newtheorem{lem}[thm]{\protect\lemmaname}

\renewcommand*{\@fnsymbol}[1]{\@arabic{#1}}

\makeatother

\usepackage{babel}
\providecommand{\assumptionname}{Assumption}
\providecommand{\definitionname}{Definition}
\providecommand{\examplename}{Example}
\providecommand{\lemmaname}{Lemma}
\providecommand{\propositionname}{Proposition}
\providecommand{\theoremname}{Theorem}

\begin{document}
\title{Accelerating invasions along an environmental gradient}
\date{\date{}}
\author{\textbf{Gwenaël Peltier}\thanks{IMAG, Univ. Montpellier, CNRS, Montpellier, France. E-mail: gwenael.peltier@umontpellier.fr}}

\maketitle
\tableofcontents{}
\begin{abstract}
We consider a population structured by a space variable and a phenotypical
trait, submitted to dispersion, mutations, growth and nonlocal competition.
This population is facing an environmental gradient: the optimal trait
for survival depends linearly on the spatial variable. The survival
or extinction depends on the sign of an underlying principal eigenvalue.
We investigate the survival case when the initial data satisfies a
so-called heavy tail condition in the space-trait plane. Under these
assumptions, we show that the solution propagates in the favorable
direction of survival by accelerating. We derive some precise estimates
on the location of the level sets corresponding to the total population
in the space variable, regardless of their traits. Our analysis also
reveals that the orientation of the initial heavy tail is of crucial
importance.\\
\\
\uline{Key Words:} structured population, nonlocal reaction-diffusion
equation, propagation, accelerating fronts.\newpage{}
\end{abstract}

\section{Introduction}

In this paper we study the propagation phenomena of the solution $n(t,x,y)$
to the following nonlocal parabolic Cauchy problem

\begin{equation}
\begin{cases}
\partial_{t}n-\partial_{xx}n-\partial_{yy}n=\left(r(y-Bx)-{\displaystyle \int_{\mathbb{R}}}K(t,x,y,y^{\prime})n(t,x,y^{\prime})dy^{\prime}\right)n, & t>0,\,(x,y)\in\mathbb{R}^{2},\\
\vspace{-0.2cm} & \vspace{-0.2cm}\\
n(0,x,y)=n_{0}(x,y), & (x,y)\in\mathbb{R}^{2}.
\end{cases}\label{eq:dim_2_xy}
\end{equation}
We shall prove that if the initial data $n_{0}\geq0$ has a heavy
tail, in a sense to be precised later, then any solution of \eqref{eq:dim_2_xy}
either goes extinct, or spreads in the favorable direction $y-Bx=0$
by \textit{accelerating}.

Equation \eqref{eq:dim_2_xy} arises in some population dynamics models,
see \cite{MirRao_13,Pre_04}. In this context $n(t,x,y)$ represents
a density of population at each time $t\geq0$, structured by a space
variable $x\in\mathbb{R}$ and a phenotypical trait $y\in\mathbb{R}$.
This population is subject to four biological processes : migration,
mutations, growth and competition. The diffusion operators $\partial_{xx}n$
and $\partial_{yy}n$ account for migration and mutations respectively.
The growth rate of the population is given by $r(y-Bx)$, where $r$
is negative outside a bounded interval. This corresponds to a population
facing an environmental gradient : to survive at location $x$, an
individual must have a trait close to the optimal trait $y_{opt}=Bx$
with $B>0$. Thus, for invasion to occur, the population has to adapt
during migration. As a consequence, it is expected that the population,
if it survives, remains confined in a strip around the optimal line
$y-Bx=0$, where $r$ is typically positive. Finally, we consider
a logistic regulation of the population density that is local in the
spatial variable and nonlocal in the trait. In other words, we consider
that, at each location, there exists an intra-specific competition
(for e.g. food) between all individuals, regardless of their traits.

The well-posedness of a Cauchy problem very similar to \eqref{eq:dim_2_xy},
but on a bounded domain, has been investigated in \cite[Theorem I.1]{Pre_04}.
As mentioned in \cite{AlfBerRao_17}, we believe the arguments could
be adapted in our context to show the existence of a global solution
on an unbounded domain. A limiting argument would then provide the existence
of solutions to \eqref{eq:dim_2_xy} in the whole domain, thanks to
the estimates on the tails of the solutions obtained in Lemma \ref{lem:n_estimates}.
However, in this article, our main interest lies in the qualitative
properties of the solutions.

\subparagraph{Survival vs extinction.}

As it is well known, the survival or extinction of the population (starting
from a localized area) depends on the sign of the generalized principal
eigenvalue $\lambda_{0}$ of the elliptic operator $-\partial_{xx}n-\partial_{yy}n-r(y-Bx)n$.
If $\lambda_{0}>0$, the population goes extinct exponentially fast
in time at rate $-\lambda_{0}$. If $\lambda_{0}<0$, the model, which
is of Fisher-KPP type \cite{Fis_37,KolPetPis_37}, satisfies the Hair-Trigger
Effect : any nonnegative initial data $n_{0}\not\equiv0$ leads to
the survival of the population and its spreading to the whole space.
Since we are concerned with propagation results, we shall assume that
$\lambda_{0}<0$ and $n_{0}\not\equiv0$ in the rest of this introduction.

\subparagraph{The local case.}

When the competition term in \eqref{eq:dim_2_xy} is replaced by a
local (in $x$ and $y$) regulation, the equation satisfies the comparison
principle and, moreover, one can assume $B=0$ without loss of generality
(through a rotation of coordinates). This yields to one of the models
considered in \cite{BerCha_12}, where the authors show the existence
of travelling waves $\varphi(x-ct,y)$ solutions of the equation for
speeds $c$ greater than or equal to a critical value $c^{\ast}>0$.
Besides, there exists a unique positive stationary solution, which
depends only on $y\in\mathbb{R}$, denoted here by $S(y)$, and all
travelling waves connect the state $S(y)$ when $x\rightarrow-\infty$
to zero when $x\rightarrow+\infty$.

One of the main results in \cite{BerCha_12} concerns the Cauchy problem.
When the initial data is compactly supported and satisfies $n_{0}(x,y)\leq S(y)$,
the solution $n(t,x,y)$ converges locally uniformly in $x$ towards
$S(y)$, while propagating at speed $c^{\ast}$ in both directions
$x\rightarrow\pm\infty$.

\subparagraph{The nonlocal but decoupled ($B=0$) case.}

It is worth mentioning that the model \eqref{eq:dim_2_xy} when $B=0$
has been analyzed in \cite{BerJinSyl_14}. In this context, one can
decouple variables $x$ and $y$, leading to a sequence of scalar
Fisher-KPP equations, each obtained by projection on each eigenfunction
of the elliptic operator $-\partial_{yy}-r(y)$. This technique \cite{BerJinSyl_14}
allows to prove, again, the existence of fronts $\varphi(x-ct,y)$
for speeds $c\geq c^{\ast}=2\sqrt{-\lambda_{0}}$, as well as the
survival of the population and its spreading at speed $c^{\ast}$
for compactly supported initial data. 

Akin to the local model in \cite{BerCha_12}, there exists a unique
positive stationary state $S=S(y)$ towards which the solution $n(t,x,y)$
converges. When $K\equiv1$, the state $S$ is actually the principal
eigenfunction associated to $\lambda_{0}$ (with a unique choice of
a multiplicative constant for $S$ to be a positive steady state).
For example, when $r$ is quadratic, $S$ is gaussian since it satisfies
the equation of the harmonic oscillator.

\subparagraph{The model \eqref{eq:dim_2_xy}.}

The propagation phenomenon of problem \eqref{eq:dim_2_xy} has been
investigated in \cite{AlfBerRao_17,AlfCovRao_12}. Notice that the
authors in \cite{AlfBerRao_17} also allow the environmental gradient
to be shifted (say by Global Warming) at a given forced speed. Since
$B\neq0$, the decoupling argument \cite{BerJinSyl_14} cannot be
invoked here, which makes the analysis more involved. As far as travelling
waves are concerned, problem \eqref{eq:dim_2_xy} admits fronts of
the form $\varphi(x-ct,y-Bx)$ only for speeds $c\geq c^{\ast}=2\sqrt{\frac{-\lambda_{0}}{1+B^{2}}}$.
However, the construction of those waves relies on a topological degree
argument \cite{AlfCovRao_12} and little is known about their behavior
for $x\rightarrow-\infty$, with $y-Bx$ being constant. This is caused
by the presence of a nonlocal competition term in \eqref{eq:dim_2_xy},
which prevents the equation to enjoy the comparison principle, see
also \cite{AlfCov_12,BerNadPerRyz_09} for similar issues related
to the scalar nonlocal Fisher-KPP equation.

The results \cite{AlfBerRao_17} for the Cauchy problem \eqref{eq:dim_2_xy}
are the following~: if $n_{0}$ is compactly supported, the total
population at $(t,x)$, given by $N(t,x)=\int_{\mathbb{R}}n(t,x,y)dy$,
spreads at speed $c^{\ast}$. While the convergence of the solution
towards a possible steady state remains an open question, it is worth
pointing out that the population density remains mainly concentrated
around the optimal trait $y=Bx$.

\subparagraph{Accelerating invasions in the one-dimensional case.}

Before going further, let us here consider the scalar Fisher-KPP equation
\cite{Fis_37,KolPetPis_37}, say, for simplicity,
\[
\partial_{t}u-\partial_{xx}u=ru(1-u),\qquad t>0,\,x\in\mathbb{R},
\]
for some $r>0$. A result from Hamel and Roques \cite{HamRoq_10}
shows that if the initial data displays a heavy tail, i.e. decays
more slowly than any exponentially decaying function, then the population
invades the whole space by accelerating. This is in contradistinction
with the well-studied case of exponentially bounded initial data,
where the level sets of the solution spread at finite speed, see\textbf{
}\cite{AroWei_78,Uch_78}. The authors also derive sharp estimates
of the location, at large time, of these level sets \cite{HamRoq_10}.

\subparagraph{Accelerating invasions in related models.}

Let us mention that acceleration also occurs for compactly supported
initial data if the diffusion term is replaced with a fractional laplacian
$-(-\partial_{xx})^{\alpha}u$, $0<\alpha<1$, see \cite{CabRoq_13},
or with a convolution term $J\ast u-u$ where the kernel $J=J(x)$
admits a heavy tail in both directions $x\rightarrow\pm\infty$ \cite{Gar_11}.
The latter case corresponds to models of population dynamics with
long-distance dispersal.

The so-called cane-toad equation proposed in \cite{BenCalMeuVoi_12} is another biological invasion model where acceleration may occur. When the trait space is unbounded, propagation of the level sets of order $O(t^{3/2})$ has been predicted in \cite{BouCalMeuMirPerRao_12}. This was then proved rigorously in \cite{BerMouRao_15} with a local competition term, and in \cite{BouHenRhy_16} for both local and nonlocal (in trait) competition, using probabilistic and analytic arguments respectively. Notice however that acceleration is not induced here by initial heavy tails, but by a phenotype-dependent term before the spatial diffusion. 

\subparagraph{Accelerating invasions in model \eqref{eq:dim_2_xy}.}

Our aim is to study accelerating invasions in model \eqref{eq:dim_2_xy},
that is to determine if acceleration occurs if $n_{0}$ displays a
``heavy tail'', a notion that needs to be precised in the two-dimensional
framework. First, we prove that if $n_{0}$ has a ``heavy tail''
in the favorable direction $y-Bx=0$, then acceleration of the invasion
occurs. Second, we derive precise estimates for the large-time location
of the level sets of the solution. Finally, we also address the case
where the ``heavy tail'' of $n_{0}$ is not positioned along the
direction $y-Bx=0$: in this case, since ill-directed, the heavy tail
does not induce acceleration.

\section{Assumptions and main results}

\subsection{Functions $r$, $K$ and $n_{0}$}

Throughout the paper, we make the following assumption.
\begin{assumption}
\label{assu:func_r_k_n0}The function $r(\cdot)\in L_{loc}^{\infty}(\mathbb{R})$
is \textup{confining}, in the sense that, for all $\delta>0$, there
exists $R>0$ such that
\begin{equation}
r(z)\leq-\delta,\qquad\text{for almost all }z\text{ such that }|z|\geq R.\label{eq:r_condition}
\end{equation}
Additionally, there exists $r_{max}>0$ such that $r(z)\leq r_{max}$
almost everywhere. 

The function $K\in L^{\infty}((0,\infty)\times\mathbb{R}^{3})$ satisfies
\begin{equation}
k_{-}\leq K\leq k_{+},\qquad\text{a.e. on }(0,+\infty)\times\mathbb{R}^{3},\label{eq:K_bounds}
\end{equation}
for some $0<k_{-}\leq k_{+}$.

Moreover, the initial data $n_{0}$ is non identically zero, and there
exist $C_{0}>0$ and $\kappa_{0}>0$ such that 
\begin{equation}
0\leq n_{0}(x,y)\leq C_{0}e^{-\kappa_{0}|y-Bx|},\qquad\text{for almost all }(x,y)\in\mathbb{R}^{2}.\label{eq:borne_n0}
\end{equation}
\end{assumption}

An enlightening example of such function $r$ is given by $r(z)=1-Az^{2}$,
hence
\begin{equation}
r(y-Bx)=1-A(y-Bx)^{2},\label{eq:example_r}
\end{equation}
for some $A>0$. Notice that the width of the strip where $r$ is
nonnegative, that is the favorable region, is $\frac{2}{\sqrt{A(1+B^{2})}}$
(see Figure \ref{fig:2D_HT_support}). As a result, both parameters
$A$ and $B>0$ play a critical role to determine whether the population
goes extinct or survives.

Condition \eqref{eq:borne_n0} allows us to obtain estimates of the
tails of $n(t,x,y)$ in the direction $y-Bx\rightarrow\pm\infty$
as given by Lemma \ref{lem:n_estimates}. Note that condition \eqref{eq:borne_n0}
is not the aforementioned heavy tail condition, for any compactly
supported function satisfies it. Before stating our heavy tail condition,
we first need to consider some spectral problems.

\subsection{Some eigenelements\label{subsec:Eigen_intro}}

As in \cite[Section 4]{AlfBerRao_17}, rather than working in the
$(x,y)$ variables, let us write 
\[
n(t,x,y)=v(t,X,Y),
\]
where $X$ (resp. $Y$) represents the direction of (resp. the direction
orthogonal to) the optimal trait $y=Bx$, that is
\begin{equation}
X=\frac{x+By}{\sqrt{1+B^{2}}},\qquad Y=\frac{y-Bx}{\sqrt{1+B^{2}}}.\label{eq:XY_eq_xy}
\end{equation}
In these new variables, equation \eqref{eq:dim_2_xy} is recast 
\begin{equation}
\partial_{t}v-\partial_{XX}v-\partial_{YY}v=\left(\tilde{r}(Y)-\int_{\mathbb{R}}K(t,\chi,\psi,y^{\prime})v\left(t,\chi,\psi\right)dy^{\prime}\right)v,\label{eq:dim_2_XY}
\end{equation}
where we use the shortcuts
\[
\tilde{r}(Y)\coloneqq r\left(\sqrt{1+B^{2}}Y\right),\qquad\chi=\chi(X,Y,y^{\prime})\coloneqq\frac{\frac{X-BY}{\sqrt{1+B^{2}}}+By^{\prime}}{\sqrt{1+B^{2}}},\qquad\psi=\psi(X,Y,y^{\prime})\coloneqq\frac{-B\frac{X-BY}{\sqrt{1+B^{2}}}+y^{\prime}}{\sqrt{1+B^{2}}}.
\]
We also note $v_{0}(X,Y)=n_{0}(x,y)$ the initial data in the new
variables.

Next, as recalled in subsection \ref{subsec:Eigen_results}, we are
equipped with a generalized principal eigenvalue $\lambda_{0}\in\mathbb{R}$
and a generalized principal eigenfunction $\Gamma_{0}\in H_{loc}^{2}(\mathbb{R})$
satisfying
\begin{equation}
\begin{cases}
-\partial_{YY}\Gamma_{0}(Y)-\tilde{r}(Y)\Gamma_{0}(Y)=\lambda_{0}\Gamma_{0}(Y) & \text{for all }Y\in\mathbb{R},\\
\Gamma_{0}>0,\qquad\text{\ensuremath{\left|\left|\Gamma_{0}\right|\right|}}_{L^{\infty}(\mathbb{R})}=1.
\end{cases}\label{eq:Eigen_1D}
\end{equation}
It is worth noting that, in the particular case where $r$ is given
by \eqref{eq:example_r}, expression \eqref{eq:Eigen_1D} corresponds
to the harmonic oscillator, for which these eigenelements can be explicitly
computed as
\begin{equation}
\lambda_{0}=\sqrt{A(1+B^{2})}-1,\qquad\Gamma_{0}(Y)=\exp\left(-\frac{1}{2}\sqrt{A(1+B^{2})}Y^{2}\right).\label{eq:gamma0_ex}
\end{equation}

Finally, for $R>0$, let us consider $\lambda_{0}^{R}$, $\Gamma_{0}^{R}(Y)$
the principal eigenelements solving the Dirichlet problem on $(-R,R)$
\begin{equation}
\begin{cases}
-\partial_{YY}\Gamma_{0}^{R}(Y)-\tilde{r}(Y)\Gamma_{0}^{R}(Y)=\lambda_{0}^{R}\Gamma_{0}^{R}(Y) & \text{for }Y\in(-R,R),\\
\Gamma_{0}^{R}(Y)=0 & \text{for }Y=\pm R,\\
\Gamma_{0}^{R}(Y)>0 & \text{for }Y\in(-R,R),\\
||\Gamma_{0}^{R}||_{\infty}=1.
\end{cases}\label{eq:Eigen_R}
\end{equation}
As recalled in Proposition \ref{prop:monotony_eig}, there holds $\lambda_{0}^{R}\searrow\lambda_{0}$
as $R\rightarrow+\infty$.

\subsection{The heavy tail condition}

We can now turn to the two-dimensional heavy tail condition. First,
we recall or introduce some definitions for one-dimensional functions.
A function $w\colon\mathbb{R}\rightarrow\mathbb{R}$ is said to be
asymptotically front-like if 
\begin{equation}
\liminf_{-\infty}w>0,\quad w>0,\quad\lim_{+\infty}w=0.\label{eq:asymp_frontL}
\end{equation}
A positive function $w\colon\mathbb{R}\rightarrow\mathbb{R}$ is said
to have a heavy tail in $+\infty$ if
\begin{equation}
\lim_{x\rightarrow+\infty}w(x)e^{\varepsilon x}=+\infty,\qquad\forall\varepsilon>0,\qquad\qquad\text{(1D Heavy Tail)}.\label{eq:1D_heavy_tail}
\end{equation}
Typical examples are ``lighter heavy tails'' \eqref{eq:light_heavy_tail},
algebraic tails \eqref{eq:algebraic_tail}, and ``very heavy tails''
\eqref{eq:heavy_heavy_tail}, that is

\vspace{-0.4cm}

\begin{alignat}{3}
 & w(x)\sim Ce^{-bx^{a}}, & \quad & \text{as }x\rightarrow+\infty, & \qquad & \text{with }C,b>0\text{ and }a\in(0,1),\label{eq:light_heavy_tail}\\
 & w(x)\sim Cx^{-a}, & \quad & \text{as }x\rightarrow+\infty, & \qquad & \text{with }C,a>0,\label{eq:algebraic_tail}\\
 & w(x)\sim C(\ln x)^{-a}, & \quad & \text{as }x\rightarrow+\infty, & \qquad & \text{with }C,a>0.\label{eq:heavy_heavy_tail}
\end{alignat}

We now state our two-dimensional heavy tail condition for equation
\eqref{eq:dim_2_xy}. Note that this condition is expressed in the
new variables, thus it applies to $v_{0}$.
\begin{assumption}
[2D heavy tail condition]\label{assu:heavy_tail_v0}Let us consider
the new coordinates $(X,Y)$ given by \eqref{eq:XY_eq_xy}. The initial
data $v_{0}(X,Y)=n_{0}(x,y)$ is such that there exists $\underline{u}_{0}\in L^{\infty}(\mathbb{R})$
satisfying \eqref{eq:asymp_frontL}-\eqref{eq:1D_heavy_tail}, so
that

\begin{equation}
v_{0}(X,Y)\geq\underline{u}_{0}(X)\boldsymbol{1}_{[\sigma_{-},\sigma_{+}]}(Y),\qquad\qquad\emph{\text{(2D Heavy Tail)}}.\label{eq:2D_heavy_tail}
\end{equation}
for some reals $\sigma_{-}<\sigma_{+}$.
\end{assumption}

Let us emphasize that we do not assume that $0\in(\sigma_{-},\sigma_{+})$,
meaning the initial data may not overlap the optimal trait line (see
Figure \ref{fig:2D_HT_support}). Moreover, the interval $(\sigma_{-},\sigma_{+})$
may not only be arbitrarily far from zero, but also arbitrarily small.
The key assumption is the correct orientation of the heavy tail, that
is in the direction $X\rightarrow+\infty$, as highlighted by subsection
\ref{subsec:Bad_HT_results}. In the survival case $\lambda_{0}<0$
and under Assumption \ref{assu:heavy_tail_v0}, we shall prove that
the solution of \eqref{eq:dim_2_xy} is accelerating.\\

\begin{figure}
\centerline{\includegraphics[scale=0.47]{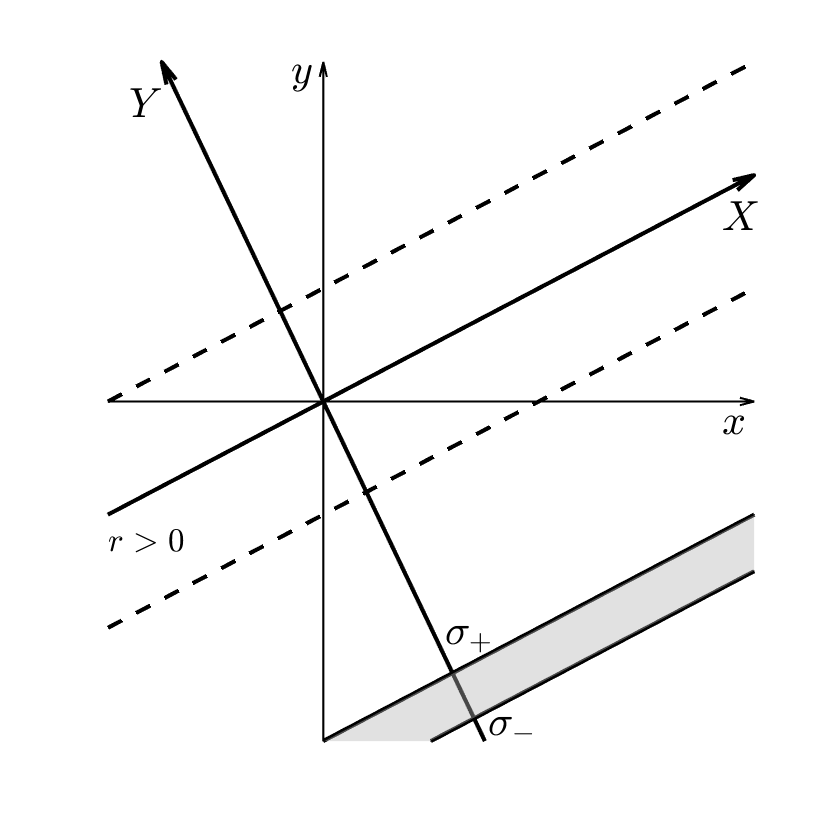}}

\caption{\label{fig:2D_HT_support}In grey, the region where $v_{0}(X,Y)$
is assumed to be greater than $\underline{u}_{0}(X)$, that is a one-dimensional
heavy tail in the $X$ direction. The dotted lines delimit the area
where $r>0$.}

\end{figure}
We now aim at providing a precise estimate of the location of the
level sets at large times. To do so, we first introduce the following
definition.
\begin{defn}
\label{def:condition_Q}A function $w$ is said to satisfy the condition
$(Q)$ if
\[
\begin{cases}
w\in L^{\infty}(\mathbb{R})\text{ and is uniformly continuous},\\
\liminf_{-\infty}w>0,\quad w>0,\quad\lim_{+\infty}w=0,\\
\exists\xi_{0}\in\mathbb{R}\text{ such that }w\text{ is }C^{2}\text{ and nonincreasing on }[\xi_{0},+\infty),\\
w^{\prime\prime}(x)=o(w(x))\text{ as }x\rightarrow+\infty,
\end{cases}\tag{\ensuremath{Q}}
\]
where $o$ denotes the Landau symbol \lq\lq little-o\rq\rq.
\end{defn}

Notice that any function $w$ satisfying $(Q)$ also satisfies $w^{\prime}(x)=o(w(x))$
as $x\rightarrow+\infty$, and thus displays a one-dimensional heavy
tail in $+\infty$, see \cite{HamRoq_10}. For the scalar Fisher-KPP
equation, when the initial data satisfies $(Q)$, the authors in \cite{HamRoq_10}
derived precise estimates on the location of the level sets of the
solution. In our context, we make the following assumption on the
initial data.
\begin{assumption}
[($Q$)-Initial bounds]\label{assu:encad_v0}Let us consider the new
coordinates $(X,Y)$ given by \eqref{eq:XY_eq_xy}. The initial data
$v_{0}(X,Y)=n_{0}(x,y)$ is such that there exist functions $\overline{u}_{0}$,
$\underline{u}_{0}$ satisfying $(Q)$ so that 
\begin{equation}
\underline{u}_{0}(X)\boldsymbol{1}_{[\sigma_{-},\sigma_{+}]}(Y)\leq v_{0}(X,Y)\leq\overline{u}_{0}(X)\Gamma_{0}(Y),\label{eq:encad_v0}
\end{equation}
for some reals $\sigma_{-}<\sigma_{+}$.
\end{assumption}

In particular, if the initial data satisfies Assumption \ref{assu:encad_v0},
then it satisfies Assumption \ref{assu:heavy_tail_v0}. As far as
the $Y$ direction is concerned, when $r$ is of the form \eqref{eq:example_r},
the eigenfunction $\Gamma_{0}$ is given by \eqref{eq:gamma0_ex},
so that \eqref{eq:encad_v0} amounts to a gaussian control on the
initial data. In the general case of a confining growth function \eqref{eq:r_condition},
one can prove that $\Gamma_{0}(Y)$ decays at least exponentially
when $|Y|\rightarrow+\infty$, see subsection \ref{subsec:Eigen_results}.
Under Assumption \ref{assu:encad_v0}, we shall derive some precise
estimates on the large-time position of the level sets, see Theorem
\ref{thm:gw_main}. 

\subsection{The extinction case}

As we shall see, under Assumption \ref{assu:func_r_k_n0}, the population
either goes extinct or survives depending on the sign of the principal
eigenvalue $\lambda_{0}$. In this short section we simply expose
the result of \cite{AlfCovRao_12}, which covers the case $\lambda_{0}>0$.
\begin{prop}
[Extinction case \cite{AlfCovRao_12}]\label{prop:extinction}Assume
$\lambda_{0}>0$. Let $r,K,n_{0}$ satisfy Assumption \ref{assu:func_r_k_n0}.
Suppose that there is $k>0$ such that
\[
n_{0}(x,y)\leq k\Gamma_{0}\left(\frac{y-Bx}{\sqrt{1+B^{2}}}\right).
\]
Then any global nonnegative solution of \eqref{eq:dim_2_xy} satisfies
\begin{equation}
n(t,x,y)\leq k\Gamma_{0}\left(\frac{y-Bx}{\sqrt{1+B^{2}}}\right)e^{-\lambda_{0}t},\label{eq:extinction_result}
\end{equation}
which implies $||n(t,\cdot,\cdot)||_{L^{\infty}(\mathbb{R}^{2})}=O(e^{-\lambda_{0}t})$,
that is an exponentially fast extinction.
\end{prop}

The proof of Proposition \ref{prop:extinction} is elementary as $n(t,x,y)$
and the right-hand side of \eqref{eq:extinction_result} are respectively
subsolution and supersolution of the parabolic operator $\partial_{t}n-\partial_{xx}n-\partial_{yy}n-r(y-Bx)n$.
The maximum principle yields the result.

\subsection{Main result : acceleration in the invasion case\label{subsec:Main_results}}

We now investigate the case where the principal eigenvalue $\lambda_{0}$
is negative. In order to capture the spreading speed of the population
in the space variable, we look at the evolution of the total population
in $(t,x)$, regardless of their trait. Thus, for any $\mu>0$, we
define the level set of $n$ by 
\[
E_{\mu}^{n}(t)=\left\{ x\in\mathbb{R}\,\Bigl|\,\int_{\mathbb{R}}n(t,x,y)dy=\mu\right\} .
\]

Let us emphasize again that, because of the nonlocal competition term,
problem \eqref{eq:dim_2_xy} does not enjoy the comparison principle.
In such situation, and as mentioned in the introduction, the behavior
``behind the front'' is typically out of reach, see \cite{AlfBerRao_17,AlfCov_12,AlfCovRao_12,BerNadPerRyz_09,FayHol_15}.
For such a reason, we are mainly interested in the spreading properties
of $E_{\mu}^{n}(t)$ for \textit{small} values $\mu$.

Let us recall that under Assumption \ref{assu:func_r_k_n0}, if $\lambda_{0}<0$
and if $n_{0}\not\equiv0$ has compact support, then the population
survives and the solution propagates at speed $c^{\ast}=2\sqrt{\frac{-\lambda_{0}}{1+B^{2}}}$,
see \cite[Theorem 4.2]{AlfBerRao_17}. Our first result shows that
there is acceleration when, instead of being compactly supported,
the initial data admits a heavy tail in the $X$ direction, in the
sense given by Assumption \ref{assu:heavy_tail_v0}.
\begin{thm}
[2D initial heavy tail implies acceleration]\label{thm:gw_HT_accel}Assume
$\lambda_{0}<0$. Let $r,K,n_{0}$ satisfy Assumptions \ref{assu:func_r_k_n0}
and \ref{assu:heavy_tail_v0}. Let $n$ be any global nonnegative
solution of \eqref{eq:dim_2_xy}. Then there exists $\beta>0$ such
that for any $\mu\in(0,\beta)$, there holds 
\[
\frac{1}{t}\min E_{\mu}^{n}(t)\rightarrow+\infty,\qquad\text{as }t\rightarrow+\infty.
\]
\end{thm}

In other words, if the initial data is greater than or equal to a front-like
function with a heavy tail in the direction $X\rightarrow+\infty$,
the solution is accelerating. We will only give a sketch of the proof
in subsection \ref{subsec:HT_accel_proof}, as it is similar to the
proof of Theorem \ref{thm:gw_main} below. 

We now state our main result, namely Theorem \ref{thm:gw_main}, which
is an accurate estimate of the position of the accelerating level
sets under Assumption \ref{assu:encad_v0}. In the rest of this article,
for any function $f\colon\mathbb{R}\rightarrow\mathbb{R}$, we denote
$f^{-1}(a)$ the set $\{x\in\mathbb{R}\mid f(x)=a\}$. 
\begin{thm}
[Asymptotic position of the accelerating level sets]\label{thm:gw_main}Assume
$\lambda_{0}<0$. Let $r,K,n_{0}$ satisfy Assumptions \ref{assu:func_r_k_n0}
and \ref{assu:encad_v0}. Let $R>0$ be large enough such that $\lambda_{0}^{R}<0$
(see subsection \ref{subsec:Eigen_intro}). Let $n$ be any global
nonnegative solution of \eqref{eq:dim_2_xy}.

Then there exists $\beta>0$ so that for any $\mu\in(0,\beta)$, $\varepsilon\in(0,-\lambda_{0}^{R})$,
$\Gamma>0$ and $\gamma>0$, there exists $T^{\ast}=T_{\mu,\varepsilon,\gamma,\Gamma,R}^{\ast}\geq0$
such that for all $t\geq T^{\ast}$, the set $E_{\mu}^{n}(t)$ is
nonempty, compact, and satisfies
\begin{equation}
E_{\mu}^{n}(t)\subset\frac{1}{\sqrt{1+B^{2}}}\left[\min\underline{u}_{0}^{-1}\left(\Gamma e^{-(-\lambda_{0}^{R}-\varepsilon)t}\right),\max\overline{u}_{0}^{-1}\left(\gamma e^{-(-\lambda_{0}+\varepsilon)t}\right)\right].\label{eq:lvl_set_estimate}
\end{equation}
\end{thm}

Let us make some comments on this theorem. Observe first that for
$t$ large enough there holds 
\begin{align*}
 & \Gamma e^{-(-\lambda_{0}^{R}-\varepsilon)t}\in\left(0,\liminf_{-\infty}\underline{u}_{0}\right),\\
 & \gamma e^{-(-\lambda_{0}+\varepsilon)t}\in\left(0,\liminf_{-\infty}\overline{u}_{0}\right),
\end{align*}
thus the sets $\underline{u}_{0}^{-1}\left(\Gamma e^{-(-\lambda_{0}^{R}-\varepsilon)t}\right)$
and $\overline{u}_{0}^{-1}\left(\gamma e^{-(-\lambda_{0}+\varepsilon)t}\right)$
are non-empty and bounded. Additionally, Assumption \ref{assu:encad_v0}
implies that $\underline{u}_{0}\leq C\overline{u}_{0}$ with $C=\min_{(\sigma_{-},\sigma_{+})}\Gamma_{0}>0$.
In conjunction with $\lambda_{0}<\lambda_{0}^{R}$, it follows that
for $t$ possibly even larger there holds 
\[
\min\underline{u}_{0}^{-1}\left(\Gamma e^{-(-\lambda_{0}^{R}-\varepsilon)t}\right)<\max\overline{u}_{0}^{-1}\left(\gamma e^{-(-\lambda_{0}+\varepsilon)t}\right),
\]
giving a meaning to \eqref{eq:lvl_set_estimate}. 

Next, notice that, given any two values $\mu$ and $\mu^{\prime}$
in $(0,\beta)$, both level sets $E_{\mu}^{n}(t)$ and $E_{\mu^{\prime}}^{n}(t)$
are included in the same interval given by expression \eqref{eq:lvl_set_estimate}.
As a consequence, Theorem \ref{thm:gw_main} implies that for any
$\varepsilon\in(0,-\lambda_{0}^{R})$ and positive real numbers $\gamma$
and $\Gamma$, there holds
\[
\liminf_{t\rightarrow+\infty}\inf_{x\leq(1+B^{2})^{-1/2}\min\underline{u}_{0}^{-1}\left(\Gamma e^{-(-\lambda_{0}^{R}-\varepsilon)t}\right)}\int_{\mathbb{R}}n(t,x,y)dy\geq\beta,
\]
\[
\lim_{t\rightarrow+\infty}\sup_{x\geq(1+B^{2})^{-1/2}\max\overline{u}_{0}^{-1}\left(\gamma e^{-(-\lambda_{0}+\varepsilon)t}\right)}\int_{\mathbb{R}}n(t,x,y)dy=0.
\]
The upper bound of $E_{\mu}^{n}(t)$ in \eqref{eq:lvl_set_estimate}
is valid for all levels $\mu$, and only requires the upper bound
of $v_{0}$ in Assumption \ref{assu:encad_v0}. However, the lower
bound of $E_{\mu}^{n}(t)$ is valid for levels $\mu<\beta$, and only
requires the lower bound of $v_{0}$ in Assumption \ref{assu:encad_v0}.
Also note that the lower bound in \eqref{eq:lvl_set_estimate} leads
to $\frac{1}{t}\min E_{\mu}^{n}(t)\rightarrow+\infty$ when $t\rightarrow+\infty$,
thus we recover the acceleration.\\

We now give a sketch of the proof. The upper bound is much easier
to prove since the nonlocal term is nonnegative. One constructs a
supersolution of the form $\phi(t,X)\Gamma_{0}(Y)$ where $\phi$
satisfies $\partial_{t}\phi-\partial_{xx}\phi=(-\lambda_{0})\phi$
with $\phi(0,\cdot)$ displaying a heavy tail. The upper bound of
Lemma  \ref{lem:HR_gen} is still valid in this case, which leads
to the result with an adequate control of the tails of $n$. 

The proof of the lower bound is much more involved. Suppose first
that $[-R,R]\subset[\sigma_{-},\sigma_{+}]$. Then, after bounding
the nonlocal term with a refinement of a Harnack inequality, we construct
a subsolution of the form $\underline{w}(t,X,Y)=u(t,X)\Gamma_{0}^{R}(Y)$
where $u$ satisfies the Fisher-KPP equation. Therefore applying Lemma
\ref{lem:HR_gen} allows us to conclude. Note that this might not
be a subsolution if $R$ were too small, leading to $\lambda_{0}^{R}$
being possibly nonnegative. In the general case we may have $[-R,R]\not\subset[\sigma_{-},\sigma_{+}]$. In that event we construct a subsolution $\underline{v}(t,X,Y)$ for $t\in [0,1]$, such that $\underline{v}(1,X,Y)\geq\rho\underline{u}_{0}(X)\Gamma_{0}^{R}(Y)$
on $\mathbb{R}\times[-R,R]$ for some $\rho>0$. Then on $[1,+\infty)$
we consider a subsolution of the same form as $\underline{w}$,
which gives the result.

In particular, to prove acceleration under the hypothesis $v_{0}(X,Y)\geq\underline{u}_{0}(X)\boldsymbol{1}_{[\sigma_{-},\sigma_{+}]}(Y)$,
we have to use $\Gamma_{0}^{R}$ instead of $\Gamma_{0}$ in order
to construct the subsolution. Because of this, we obtain $-\lambda_{0}^{R}$
in the lower bound of \eqref{eq:lvl_set_estimate}. Had we supposed
the stronger hypothesis $v_{0}(X,Y)\geq\underline{u}_{0}(X)\Gamma_{0}(Y)$
instead, we could replace $-\lambda_{0}^{R}$ with $-\lambda_{0}$
and take any $\varepsilon\in(0,-\lambda_{0})$. Let us also mention
that $\beta$ tends to zero as $R\rightarrow+\infty$, leading to
a trade-off. Indeed, a large value of $R$ provides a more precise
location of the level sets, but also reduces the range of level sets
being located. \\

We conclude this section by applying Theorem \ref{thm:gw_main} in
the cases where the functions $\underline{u}_{0}$ and $\overline{u}_{0}$
are of the forms \eqref{eq:light_heavy_tail}-\eqref{eq:heavy_heavy_tail}.
For simplicity, we only consider the lower bound. 
\begin{example}
Suppose there exist $X_{0},b>0$ and $a\in(0,1)$ such that $\overline{u}_{0}(X)=Ce^{-bx^{a}}$
on $[X_{0},+\infty)$. Then if we select $\Gamma=C$, the lower bound
in \eqref{eq:lvl_set_estimate} becomes
\[
\min E_{\mu}^{n}(t)\geq\frac{1}{\sqrt{1+B^{2}}}\left(\frac{1}{b}(-\lambda_{0}^{R}-\varepsilon)t\right)^{1/a},
\]
meaning the total population spreads with at least algebraic, superlinear
speed.
\end{example}

\begin{example}
Suppose there exist $X_{0},C,a>0$ such that $\underline{u}_{0}(X)=CX^{-a}$
on $[X_{0},+\infty)$. Then\textbf{ }if we select $\Gamma=C$, the
lower bound in \eqref{eq:lvl_set_estimate} becomes
\[
\min E_{\mu}^{n}(t)\geq\frac{1}{\sqrt{1+B^{2}}}\exp\left(\frac{1}{a}(-\lambda_{0}^{R}-\varepsilon)t\right),
\]
thus the total population spreads with at least exponential speed.
\end{example}

\begin{example}
Suppose there exist $X_{0}>1$ and $C,a>0$ such that $\overline{u}_{0}(X)=C(\ln x)^{-a}$
on $[X_{0},+\infty)$. Then if we select $\Gamma=C$, the lower bound
in \eqref{eq:lvl_set_estimate} becomes
\[
\min E_{\mu}^{n}(t)\geq\frac{1}{\sqrt{1+B^{2}}}\exp\left(\exp\left(\frac{1}{a}(-\lambda_{0}^{R}-\varepsilon)t\right)\right),
\]
that is the total population spreads with at least superexponential
speed.
\end{example}

\subsection{When the heavy tail is ill-directed\label{subsec:Bad_HT_results}}

When the initial data admits a heavy tail in direction $X\rightarrow+\infty$,
in the sense of Assumption \ref{assu:heavy_tail_v0}, Theorem \ref{thm:gw_HT_accel}
proves the acceleration of the propagation. It is worth wondering
if acceleration still occurs when considering heavy tail initial condition
in a different direction than $X\rightarrow+\infty$. For the sake
of clarity, we only consider the direction $x\rightarrow+\infty$,
but the proof is easily adapted to any direction 
\[
X^{\prime}=\frac{x+B^{\prime}y}{\sqrt{1+B^{\prime2}}}\rightarrow+\infty,\quad\text{with }B^{\prime}\neq B.
\]

\begin{thm}
[Ill-directed heavy tail prevents acceleration]\label{thm:gw_HTx}Suppose
$\lambda_{0}<0$. Suppose $r,K$ satisfy Assumption \ref{assu:func_r_k_n0}.
Suppose $n_{0}$ satisfies
\begin{equation}
0\leq n_{0}(x,y)\leq u_{0}(x)\boldsymbol{1}_{[\sigma_{-},\sigma_{+}]}(y),\label{eq:borne_n0_HTx}
\end{equation}
where $u_{0}\in L^{\infty}(\mathbb{R})$ and $\sigma_{-}<\sigma_{+}$.
Let $n$ be any global nonnegative solution of \eqref{eq:dim_2_xy}.

Then if we define $c^{\ast}\coloneqq2\sqrt{\frac{-\lambda_{0}}{1+B^{2}}}$,
there holds
\begin{equation}
\limsup_{t\rightarrow+\infty}\int_{\mathbb{R}}n(t,ct,y)dy=0,\qquad\forall|c|>c^{\ast}.\label{eq:HTx_vit_c*}
\end{equation}
\end{thm}

Notice that $u_{0}$ appearing in \eqref{eq:borne_n0_HTx} is only
assumed to be bounded. In particular, even if $u_{0}\equiv\text{cst}>0$,
a much stronger assumption than a heavy tail, acceleration does not
occur because of ill-orientation. 

Before going further, let us mention that \cite[Theorem 4.2]{AlfBerRao_17}
shows that, when $r,K$ satisfy Assumption \ref{assu:func_r_k_n0}
and $n_{0}\not\equiv0$ is compactly supported, the spreading speed
of the population is exactly $c^{\ast}$, in the sense that
\begin{align}
\limsup_{t\rightarrow+\infty}\int_{\mathbb{R}}n(t,ct,y)dy=0, & \qquad\forall|c|>c^{\ast},\label{eq:speed_leq_c*}\\
\liminf_{t\rightarrow+\infty}\int_{\mathbb{R}}n(t,ct,y)dy\geq\beta, & \qquad\forall|c|<c^{\ast},\label{eq:speed_geq_c*}
\end{align}
for some $\beta>0$ that may depend on $c$ when $|c|\rightarrow c^{\ast}$.

A consequence of Theorem \ref{thm:gw_HTx} is that if $n_{0}\not\equiv0$
satisfies \eqref{eq:borne_n0_HTx}, the population spreads exactly
at speed $c^{\ast}$, in the sense given by \eqref{eq:speed_leq_c*}-\eqref{eq:speed_geq_c*}.
To prove that \eqref{eq:speed_geq_c*} holds, one cannot invoke the
comparison principle because of the nonlocal term in \eqref{eq:dim_2_xy}.
However, an essential element of the proof of Theorem \ref{thm:gw_HTx}
is the control of the tails \eqref{eq:HTx_n_tails}. Using it, one
can adapt the proof of \cite[Theorem 4.2]{AlfBerRao_17} to show that
\eqref{eq:speed_geq_c*} is valid.

\subparagraph{Outline of the paper.}

The rest of this article is organized as follows. In Section \ref{sec:Preliminaries}
we provide some materials necessary to the proof, that is an equivalent
of Theorem \ref{thm:gw_main} for the scalar Fisher-KPP equation,
some properties of functions satisfying $(Q)$, some principal eigenelements
of elliptic operators, some estimates on the tails of $n$ as well
as a refinement of the parabolic Harnack inequality. Section \ref{sec:Main_proof}
is devoted to the proof of Theorem \ref{thm:gw_main}, and presents
a sketch of the proof of Theorem \ref{thm:gw_HT_accel}. Finally,
Section \ref{sec:HTx} addresses the proof of Theorem \ref{thm:gw_HTx}.

\section{Preliminaries\label{sec:Preliminaries}}

\subsection{Acceleration in the scalar Fisher-KPP equation}

We consider here the Fisher-KPP equation with a logistic reaction
term~:

\begin{equation}
\begin{cases}
\partial_{t}u-\partial_{xx}u=\Lambda u(1-u), & t>0,\ x\in\mathbb{R},\\
u(0,x)=u_{0}(x), & x\in\mathbb{R},
\end{cases}\label{eq:Fisher-KPP}
\end{equation}
with $\Lambda>0$. The function $u_{0}\colon\mathbb{R}\rightarrow[0,1]$
is assumed to be uniformly continuous and asymptotically front-like,
in the sense of \eqref{eq:asymp_frontL}, and to display a (one-dimensional)
heavy tail in $+\infty$, in the sense of \eqref{eq:1D_heavy_tail}.
Under these assumptions, Hamel and Roques \cite{HamRoq_10} proved
that the level sets of $u$, defined for $\eta\in(0,1)$ by
\[
E_{\eta}(t)\coloneqq\left\{ x\in\mathbb{R}\mid u(t,x)=\eta\right\} ,
\]
propagate to the right by accelerating, that is $\min E_{\eta}(t)/t\rightarrow+\infty$
as $t\rightarrow+\infty$. Under assumption $(Q)$, they also provide
sharp estimates on the position of the level sets. This result, which
will be an essential tool for our analysis, reads as follows.
\begin{lem}
[See {\cite[Theorem 1.1]{HamRoq_10}}]\label{lem:HR_gen}Let $w$
satisfy $(Q)$, see Definition \ref{def:condition_Q}. Let $u(t,x)$
be the solution of \eqref{eq:Fisher-KPP} with initial condition $u_{0}\coloneqq w/||w||_{\infty}$.
Then for any $\eta\in(0,1)$, $\varepsilon\in(0,\Lambda)$, $\gamma>0$
and $\Gamma>0$, there exists $T=T_{\eta,\varepsilon,\gamma,\Gamma,\Lambda}\geq0$
so that 
\[
E_{\eta}(t)\subset w^{-1}\left(\left[\gamma e^{-(\Lambda+\varepsilon)t},\Gamma e^{-(\Lambda-\varepsilon)t}\right]\right),\qquad\forall t\geq T.
\]
\end{lem}

In the sequel, in order to prove Theorem \ref{thm:gw_main}, we shall
construct some sub- and super-solutions of the form $u(t,X)\Gamma_{0}(Y)$
where $u(t,X)$ solves \eqref{eq:Fisher-KPP} or a linear version
of \eqref{eq:Fisher-KPP}. Then, for the estimates on $E_{\eta}(t)$
provided by Lemma \ref{lem:HR_gen} to transfer to estimates on $E_{\mu}^{n}(t)=\left\{ x\in\mathbb{R}\mid\int_{\mathbb{R}}n(t,x,y)dy=\mu\right\} $,
we shall need a technical result which we now state.
\begin{prop}
\label{prop:ppties_Q}Let $w$ satisfy $(Q)$. Then there exists $\xi_{1}>\xi_{0}$
such that $w(x)>w(\xi_{1})$ for any $x<\xi_{1}$.

In addition, for any $0<a<b$, $\Gamma_{a}>0$, $\Gamma_{b}>0$ and
$\chi>0$, there exists $t^{\ast}\geq0$ such that
\begin{equation}
\min w^{-1}\left(\Gamma_{a}e^{-at}\right)+\chi\leq\min w^{-1}\left(\Gamma_{b}e^{-bt}\right),\qquad\forall t\geq t^{\ast},\label{eq:___min_Qprop}
\end{equation}
\begin{equation}
\max w^{-1}\left(\Gamma_{a}e^{-at}\right)+\chi\leq\max w^{-1}\left(\Gamma_{b}e^{-bt}\right),\qquad\forall t\geq t^{\ast}.\label{eq:___max_Qprop}
\end{equation}
\end{prop}

\begin{proof}
Set $m\coloneqq\inf_{(-\infty,\xi_{0}]}w$. Since $\liminf_{-\infty}w>0$
and $w>0$, it is easy to check that $m>0$. Now, since $\lim_{+\infty}w=0$,
there exists $x_{+}>\xi_{0}$ such that $w(x_{+})<m$. Since $w(x_{+})>0$
and $\lim_{+\infty}w=0$, we can find $\xi_{1}\geq x_{+}$ satisfying
$w^{\prime}(\xi_{1})<0$. Finally, as $w^{\prime}(\xi_{1})<0$ and
$w$ is nonincreasing on $[\xi_{0},+\infty)$, we can readily check
that, for any $x<\xi_{1}$, there holds
\[
w(\xi_{1})\begin{cases}
\leq w(x_{+})<m\leq w(x), & \text{if }x\leq\xi_{0},\\
<w(x), & \text{if }x\in(\xi_{0},\xi_{1}),
\end{cases}
\]
which proves the first assertion.

We now turn to the second assertion. We only give a proof of \eqref{eq:___min_Qprop},
seeing as the proof of \eqref{eq:___max_Qprop} is identical. In the
first place, set $\overline{t}\geq0$ large enough such that for any
$t\geq\overline{t}$
\[
\Gamma_{a}e^{-at},\Gamma_{b}e^{-bt}\in(0,m),\qquad\forall t\geq\overline{t},
\]
hence the sets $w^{-1}\left(\Gamma_{a}e^{-at}\right)$, $w^{-1}\left(\Gamma_{b}e^{-bt}\right)$
are well-defined and compact. Next, suppose by contradiction that
there exist $0<a<b$ and positive constants $\Gamma_{a},\Gamma_{b},\chi$
such that
\[
\forall t^{\ast}\geq\overline{t},\quad\exists t\geq t^{\ast}\qquad\min w^{-1}\left(\Gamma_{a}e^{-at}\right)+\chi>\min w^{-1}\left(\Gamma_{b}e^{-bt}\right).
\]
As a result, we can construct an increasing sequence $(t_{n})_{n}$
such that $\lim_{n\rightarrow+\infty}t_{n}=+\infty$ and the above
inequality holds for $t=t_{n}$. In particular, there exists $N\in\mathbb{N}$
such that for any $n\geq N$, there holds $\Gamma_{a}e^{-at_{n}}\leq w(\xi_{1})$,
whence

\[
\min w^{-1}\left(\Gamma_{a}e^{-at_{n}}\right)\geq\xi_{1}>\xi_{0}.
\]
Meanwhile, since $b>a$, we can select $N$ possibly even larger so
that for any $n\geq N$ there holds 
\[
\Gamma_{b}e^{-bt_{n}}\leq\frac{\Gamma_{a}}{2}e^{-at_{n}}<\Gamma_{a}e^{-at_{n}}.
\]
Both assertions imply, by monotony of $w$ on $[\xi_{0},+\infty)$,
that
\[
\min w^{-1}\left(\Gamma_{a}e^{-at_{n}}\right)<\min w^{-1}\left(\Gamma_{b}e^{-bt_{n}}\right).
\]
Now, from the mean value theorem, there is $\theta_{n}\in\Bigl(\min w^{-1}\left(\Gamma_{a}e^{-at_{n}}\right),\min w^{-1}\left(\Gamma_{b}e^{-bt_{n}}\right)\Bigr)$
such that 
\[
w^{\prime}(\theta_{n})=\frac{\Gamma_{a}e^{-at_{n}}-\Gamma_{b}e^{-bt_{n}}}{\min w^{-1}\left(\Gamma_{a}e^{-at_{n}}\right)-\min w^{-1}\left(\Gamma_{b}e^{-bt_{n}}\right)}<0,
\]
therefore 
\begin{align}
\left|w^{\prime}(\theta_{n})\right| & \geq\frac{\Gamma_{a}e^{-at_{n}}-\Gamma_{b}e^{-bt_{n}}}{\chi}\geq\frac{\Gamma_{a}e^{-at_{n}}}{2\chi}.\label{eq:___absurd}
\end{align}
However, since $w$ satisfies $(Q)$, there holds $w^{\prime}(x)=o(w(x))$
as $x\rightarrow+\infty$. As a consequence, there exists $x_{\chi}\in\mathbb{R}$
such that $|w^{\prime}(x)|\leq\frac{1}{4\chi}w(x)$ for any $x\geq x_{\chi}$.
As $\lim_{n\rightarrow+\infty}\theta_{n}=+\infty$, we obtain $\theta_{n}>x_{\chi}$
for $n$ large enough. For such $n$, we derive the following inequality~:
\[
\left|w^{\prime}(\theta_{n})\right|\leq\frac{1}{4\chi}w(\theta_{n})\leq\frac{1}{4\chi}\Gamma_{a}e^{-at_{n}},
\]
which contradicts \eqref{eq:___absurd}. Thus \eqref{eq:___min_Qprop}
holds.
\end{proof}

\subsection{Some eigenvalue problems\label{subsec:Eigen_results}}

We present here some useful eigenelements. This subsection is quoted
from \cite[Subsection 2.1]{AlfBerRao_17}, which was based on the
results of \cite{BerHamRos_07,BerNirVar_94,BerRos_08}.

The theory of generalized principal eigenvalue has been developed
in \cite{BerHamRos_07}, and is well adapted to our problem when $r$,
thus $\tilde{r}$, is bounded. Following \cite{BerHamRos_07}, we
can then define, for $\tilde{r}\in L^{\infty}(\Omega)$ and $\Omega\subset\mathbb{R}$
a (possibly unbounded) nonempty domain, the generalized principal
eigenvalue 
\begin{equation}
\lambda(\tilde{r},\Omega)\coloneqq\sup\left\{ \lambda\in\mathbb{R}\mid\exists\phi\in H_{loc}^{2}(\Omega),\,\phi>0,\,\phi^{\prime\prime}(Y)+(\tilde{r}(Y)+\lambda)\phi(Y)\leq0\right\} .\label{eq:def_gen_ppal_eig}
\end{equation}
As shown in \cite{BerHamRos_07}, if $\Omega$ is bounded, $\lambda(\tilde{r},\Omega)$
coincides with the Dirichlet principal eigenvalue $\lambda_{D}$,
that is the unique real number such that there exists $\phi$ defined
on $\Omega$ (unique up to multiplication by a scalar) satisfying
\[
\begin{cases}
-\phi^{\prime\prime}(Y)-\tilde{r}(Y)\phi(Y)=\lambda_{D}\phi(Y) & \text{a.e. in }\Omega,\\
\phi>0 & \text{on }\Omega,\\
\phi=0 & \text{on }\partial\Omega.
\end{cases}
\]

Note that $\lambda(\tilde{r},\Omega)\leq\lambda(\tilde{r},\Omega^{\prime})$
if $\Omega\supset\Omega^{\prime}$. The following proposition shows
that $\lambda(\tilde{r},\Omega)$ can be obtained as a limit of increasing
domains. 
\begin{prop}
\label{prop:monotony_eig}Assume that $\tilde{r}\in L^{\infty}(\mathbb{R})$.
For any nonempty domain $\Omega\subset\mathbb{R}$ and any sequence
of nonempty domains $(\Omega_{n})_{n\in\mathbb{N}}$ such that 
\[
\Omega_{n}\subset\Omega_{n+1},\qquad\cup_{n\in\mathbb{N}}\Omega_{n}=\Omega,
\]
there holds $\lambda(\tilde{r},\Omega_{n})\searrow\lambda(\tilde{r},\Omega)$
as $n\rightarrow+\infty$. Furthermore, there exists a generalized
principal eigenfunction, that is a positive function $\Gamma\in H_{loc}^{2}(\mathbb{R})$
such that 
\[
-\Gamma^{\prime\prime}(Y)-\tilde{r}(Y)\Gamma(Y)=\lambda(\tilde{r},\Omega)\Gamma(Y),\qquad\text{a.e. in }\Omega.
\]
\end{prop}

Since our growth function $\tilde{r}$ is only assumed to be bounded
from above, we extend definition \eqref{eq:def_gen_ppal_eig} to functions
$\tilde{r}$ in $L_{loc}^{\infty}(\Omega)$ such that $\tilde{r}\leq r_{max}$
on $\Omega$, for some $r_{max}>0$. The set
\[
\Lambda(\tilde{r},\Omega)\coloneqq\left\{ \lambda\in\mathbb{R}\mid\exists\phi\in H_{loc}^{2}(\Omega),\,\phi>0,\,\phi^{\prime\prime}(Y)+(\tilde{r}(Y)+\lambda)\phi(Y)\leq0\right\} 
\]
is not empty since $\Lambda(\max(\tilde{r},-r_{max}),\Omega)\subset\Lambda(\tilde{r},\Omega)$,
and is bounded from above, thanks to the monotony property of $\Omega\mapsto\Lambda(\tilde{r},\Omega)$.
Finally, going back to the proof of \cite[Proposition 4.2]{BerHamRos_07},
we notice that Proposition \ref{prop:monotony_eig} remains valid
under the weaker assumption $\tilde{r}\in L_{loc}^{\infty}(\Omega)$
is bounded from above.

It follows from the above discussion that we are equipped with the
generalized principal eigenvalue $\lambda_{0}\in\mathbb{R}$ and a
generalized principal eigenfunction $\Gamma_{0}\in H_{loc}^{2}(\mathbb{R})$
such that
\[
\begin{cases}
-\Gamma_{0}^{\prime\prime}(Y)-\tilde{r}(Y)\Gamma_{0}(Y)=\lambda_{0}\Gamma_{0}(Y) & \text{a.e. in }\Omega,\\
\Gamma_{0}>0 & \text{on }\Omega,\\
||\Gamma_{0}||_{L^{\infty}(\mathbb{R})}=1.
\end{cases}
\]
Let us also mention that, given that $\tilde{r}$ satisfies Assumption
\ref{assu:func_r_k_n0}, the function $\Gamma_{0}$ decays at least
exponentially as $|Y|\rightarrow+\infty$. This result holds by using
the comparison principle on $\{|Y|>Y_{0}\}$ with a supersolution
of the form $Ce^{-a|Y|}$, with $Y_{0}$ large enough (so that $\tilde{r}+\lambda_{0}\leq-\varepsilon$
for some $\varepsilon>0$), $a$ small enough and $C$ large enough.

\subsection{Preliminary estimates}

The following lemma gathers preliminary results from \cite{AlfBerRao_17},
with $n_{0}$ satisfying Assumption \ref{assu:func_r_k_n0} instead
of being compactly supported. The proof of the following a priori
estimates is easily adapted from \cite[Lemmas 2.3 and 2.4]{AlfBerRao_17}
and is therefore omitted (see also the proof of Lemma \ref{lem:HTx_tails}).
\begin{lem}
[Some a priori estimates]\label{lem:n_estimates}Let $r,K,n_{0}$
satisfy Assumption \ref{assu:func_r_k_n0}. Then, there exist $N_{\infty}>0$,
$C>0$ and $\kappa>0$ such that any global nonnegative solution of
\eqref{eq:dim_2_xy} satisfies 
\begin{equation}
\int_{\mathbb{R}}n(t,x,y)dy\leq N_{\infty},\label{eq:maj_intn}
\end{equation}
\begin{equation}
n(t,x,y)\leq Ce^{-\kappa|y-Bx|},\label{eq:maj_n}
\end{equation}
for all $t\geq0$, $x\in\mathbb{R}$, $y\in\mathbb{R}$.
\end{lem}

While Lemma \ref{lem:n_estimates} provides us with a uniform bound,
we need more precise estimates on the nonlocal term $\int_{\mathbb{R}}n(t,x,y)dy$.
To do so, we invoke a refinement of the parabolic Harnack inequality,
as exposed in \cite{AlfBerRao_17}.

For any $(t,x)\in(0,+\infty)\times\mathbb{R}^{N}$ with $N\geq1$,
we consider a solution $u(t,x)$ of the following linear parabolic
equation 
\begin{equation}
\partial_{t}u(t,x)-\sum_{i,j=1}^{N}a_{i,j}(t,x)\partial_{x_{i}x_{j}}u(t,x)-\sum_{i=1}^{N}b_{i}(t,x)\partial_{x_{i}}u(t,x)=f(t,x)u(t,x),\qquad t>0,\,x\in\mathbb{R}^{N}\label{eq:___parab_pb}
\end{equation}
where the coefficients are bounded, and $(a_{i,j})_{i,j=1,\dots,N}$
is uniformly elliptic.
\begin{thm}
[{A refinement of the Harnack inequality \cite[Theorem 2.7]{AlfBerRao_17}}]\label{thm:Harnack_gen}Assume
that all the coefficients $(a_{i,j})_{i,j=1,\dots,N}$, $(b_{i})_{i=1,\dots,N}$,
f belong to $L_{loc}^{\infty}((0,+\infty)\times\mathbb{R}^{N})$,
and that $(a_{i,j})$ is uniformly positive definite on $\mathbb{R}^{N}$.
Assume there exists $K>0$ such that, for all $1\leq i,j\leq N$,
\[
a_{i,j}(t,x)\leq K,\ b_{i}(t,x)\leq K,\ f(t,x)\leq K,\qquad\text{a.e. on }(0,+\infty)\times\mathbb{R}^{N}.
\]
Let $R,\delta,U,\varepsilon,\rho$ be positive constants.

There exists $C>0$ such that for any $\overline{t}\geq\varepsilon$,
any $\overline{x}\in\mathbb{R}^{N}$ and any nonnegative weak solution
$u\in H^{1}((0,+\infty)\times\mathbb{R}^{N})$ of \eqref{eq:___parab_pb}
satisfying $||u||_{L^{\infty}(\mathbb{R}^{N})}\leq U$, there holds
\[
\max_{x\in B(\overline{x},R)}u(\overline{t},x)\leq C\min_{x\in B(\overline{x},R)}u(\overline{t},x)+\delta.
\]
\end{thm}

Notice that, as seen from the proof of \cite[Theorem 2.7]{AlfBerRao_17},
the constant $C>0$ does not depend on $\overline{t}$ provided that
$\overline{t}\geq\varepsilon>0$, which validates the above setting.

\section{Acceleration result\label{sec:Main_proof}}

Subsections \ref{subsec:Upper_bound} and \ref{subsec:Lower_bound}
are devoted to prove the following~: under the hypotheses of Theorem
\ref{thm:gw_main}, there exist $T_{\mu,\varepsilon,\gamma}^{\ast}>0$
and $T_{\mu,\varepsilon,\Gamma,R}^{\ast}>0$ such that
\begin{equation}
\int_{\mathbb{R}}n(t,x,y)dy<\mu,\qquad\forall x\geq(1+B^{2})^{-1/2}\max\overline{u}_{0}^{-1}\left(\gamma e^{-(-\lambda_{0}+\varepsilon)t}\right),\qquad\forall t\geq T_{\mu,\varepsilon,\gamma}^{\ast},\label{eq:___intn_sup}
\end{equation}
\begin{equation}
\int_{\mathbb{R}}n(t,x,y)dy>\mu,\qquad\forall x\leq(1+B^{2})^{-1/2}\min\underline{u}_{0}^{-1}\left(\Gamma e^{-(-\lambda_{0}^{R}-\varepsilon)t}\right),\qquad\forall t\geq T_{\mu,\varepsilon,\Gamma,R}^{\ast}.\label{eq:___intn_inf}
\end{equation}
Subsection \ref{subsec:Proof_conc} concludes the proof of our main
result, namely Theorem \ref{thm:gw_main}, based on \eqref{eq:___intn_sup}-\eqref{eq:___intn_inf}.
Lastly, in subsection \ref{subsec:HT_accel_proof}, we sketch the
proof of Theorem \ref{thm:gw_HT_accel}.

In the rest of this section, in view of \eqref{eq:K_bounds}, we shall
consider $K\equiv1$ without loss of generality. Additionally, to
alleviate notations, the function $\tilde{r}(Y)=r(\sqrt{1+B^{2}}Y)$
will be denoted as $r$.

\subsection{The upper bound \eqref{eq:___intn_sup}\label{subsec:Upper_bound}}

This subsection is devoted to the proof of \eqref{eq:___intn_sup}.
\begin{lem}
\label{lem:gw_borne_sup}Let $\Lambda>0$ and $\phi$ the solution
of the Cauchy problem
\begin{equation}
\begin{cases}
\partial_{t}\phi-\partial_{XX}\phi=\Lambda\phi, & t>0,\,X\in\mathbb{R},\\
\phi(0,X)=\overline{u}_{0}(X), & X\in\mathbb{R},
\end{cases}\label{eq:___heat_Cauchy_upp}
\end{equation}
where $\overline{u}_{0}$ satisfies $(Q)$. Set $E_{\eta}^{\phi}(t)=\left\{ X\in\mathbb{R}\mid\phi(t,X)=\eta\right\} $
for any $\eta>0$.

Then for any $\eta>0$, $\varepsilon\in(0,\Lambda)$, $\gamma>0$,
there exists $T=T_{\eta,\varepsilon,\gamma}\geq0$ such that for any
$t\geq T$ the set $E_{\eta}^{\phi}(t)$ is nonempty, admits a maximum,
and

\[
\max E_{\eta}^{\phi}(t)\leq\max\overline{u}_{0}^{-1}\left(\gamma e^{-(\Lambda+\varepsilon)t}\right).
\]
\end{lem}

This means that the upper bound of Theorem \ref{lem:HR_gen} is still
valid when the logistic reaction term $\Lambda u(1-u)$ in \eqref{eq:Fisher-KPP}
is replaced with $\Lambda u$. Note that, as in Lemma \ref{lem:HR_gen},
$T_{\eta,\varepsilon,\gamma}$ also depends on $\Lambda$. However,
we ignore it here since we will fix $\Lambda=-\lambda_{0}$. The proof
of Lemma \ref{lem:gw_borne_sup} is easily adapted from that of \cite[Theorem 1.1]{HamRoq_10}
and is consequently omitted.
\begin{proof}
[Proof of \eqref{eq:___intn_sup}]Set $\Lambda=-\lambda_{0}>0$,
and 
\[
\overline{v}(t,X,Y)=\phi(t,X)\Gamma_{0}(Y),
\]
where $\phi$ is the solution of the Cauchy problem \eqref{eq:___heat_Cauchy_upp}.
We readily check that $\overline{v}$ is a supersolution of the operator
$\partial_{t}-\partial_{XX}-\partial_{YY}-r(Y)$~:
\[
\partial_{t}\overline{v}-\partial_{XX}\overline{v}-\partial_{YY}\overline{v}-r(Y)\overline{v}=\left(\partial_{t}\phi-\partial_{XX}\phi+\lambda_{0}\phi\right)\Gamma_{0}=0.
\]
Meanwhile, since $v\geq0$, it is clear that $v$ is a subsolution
of the same operator. Since \eqref{eq:encad_v0} provides $v_{0}(X,Y)\leq\overline{v}(0,X,Y)$,
we conclude with the maximum principle that $v\leq\overline{v}$ on
$[0,+\infty)\times\mathbb{R}^{2}$.

Next, let $\mu>0$, $\varepsilon\in(0,\Lambda)$, and $\gamma>0$.
One can select $\eta>0$ and $\delta>0$ such that 
\begin{equation}
\eta\sqrt{1+B^{2}}\int_{\mathbb{R}}\Gamma_{0}(Y)dY+\delta<\mu.\label{eq:___<_mu}
\end{equation}
From our control of the tails \eqref{eq:maj_n}, there exist $C,\kappa>0$
such that
\[
v(t,X,Y)\leq Ce^{-\kappa\sqrt{1+B^{2}}|Y|},\qquad\forall t\geq0,\,\forall(X,Y)\in\mathbb{R}^{2}.
\]
From there, we can find $\zeta=\zeta(\delta)>0$ large enough, such
that for any $t\geq0$ and $x\in\mathbb{R}$, there holds
\begin{align*}
\int_{\mathbb{R}}n(t,x,y)dy & =\int_{\mathbb{R}}v\Bigl(t,X(x,y),Y(x,y)\Bigr)dy\\
 & =\int_{\mathbb{R}}v\left(t,\frac{x+By}{\sqrt{1+B^{2}}},\frac{y-Bx}{\sqrt{1+B^{2}}}\right)dy\\
 & =\sqrt{1+B^{2}}\int_{\mathbb{R}}v(t,\sqrt{1+B^{2}}x+Bs,s)ds\\
 & \leq\sqrt{1+B^{2}}\left[\int_{-\infty}^{-\zeta}Ce^{-\kappa\sqrt{1+B^{2}}|s|}ds+\int_{-\zeta}^{+\infty}\overline{v}(t,\sqrt{1+B^{2}}x+Bs,s)ds\right]\\
 & \leq\delta+\sqrt{1+B^{2}}\int_{-\zeta}^{+\infty}\phi(t,\sqrt{1+B^{2}}x+Bs)\Gamma_{0}(s)ds.
\end{align*}

Now, set $E_{\eta}^{\phi}(t)=\left\{ x\in\mathbb{R}\mid\phi(t,X)=\eta\right\} $.
We will show that for any
\[
x\geq\frac{1}{\sqrt{1+B^{2}}}\max\overline{u}_{0}^{-1}\left(\gamma e^{-(\Lambda+\varepsilon)t}\right),
\]
and for $t$ large enough there holds $\phi(t,\sqrt{1+B^{2}}x+Bs)\leq\eta$
for any $s\geq-\zeta$. For now, let us only assume the condition
on $x$. By applying Proposition \ref{prop:ppties_Q} to $\overline{u}_{0}$
with
\[
\begin{cases}
a=\Lambda+\varepsilon/2, & b=\Lambda+\varepsilon,\\
\Gamma_{a}=\gamma, & \Gamma_{b}=\gamma,\\
\chi=B\zeta,
\end{cases}
\]
there exists $t^{\ast}(a,b,\Gamma_{a},\Gamma_{b},\chi)=t_{\varepsilon,\gamma,\zeta}^{\ast}\geq0$
such that for all $t\geq t_{\varepsilon,\gamma,\zeta}^{\ast}$ there
holds
\begin{align*}
\max\overline{u}_{0}^{-1}\left(\gamma e^{-(\Lambda+\varepsilon/2)t}\right)+B\zeta & \leq\max\overline{u}_{0}^{-1}\left(\gamma e^{-(\Lambda+\varepsilon)t}\right).
\end{align*}
From there, Lemma \ref{lem:gw_borne_sup} proves the existence of
$T_{\eta,\varepsilon/2,\gamma}\geq0$ such that for all $t\geq T_{\mu,\varepsilon,\gamma}^{\ast}\coloneqq\max(t_{\varepsilon,\gamma,\zeta}^{\ast},T_{\eta,\varepsilon/2,\gamma})$
the following holds
\begin{align*}
\max E_{\eta}^{\phi}(t) & \leq\max\overline{u}_{0}^{-1}\left(\gamma e^{-(\Lambda+\varepsilon/2)t}\right)\\
 & \leq\max\overline{u}_{0}^{-1}\left(\gamma e^{-(\Lambda+\varepsilon)t}\right)-B\zeta\\
 & \leq\sqrt{1+B^{2}}x-B\zeta.
\end{align*}
It is then easily deduced that for any $s>-\zeta$, there holds $\phi(t,\sqrt{1+B^{2}}x+Bs)<\eta$.
Indeed, assume by contradiction that there exists $x_{0}>\sqrt{1+B^{2}}x-B\zeta$
such that $\phi(t,x_{0})\geq\eta$. Since for any $t\geq0$ one has
$\phi(t,X)\rightarrow0$ as $X\rightarrow+\infty$, there would exist
$x_{1}\in E_{\eta}^{\phi}(t)\cap[x_{0},+\infty)$, which contradicts
the above inequality.

Finally, for any $t\geq T_{\mu,\varepsilon,\gamma}^{\ast}$ and $x\geq\frac{1}{\sqrt{1+B^{2}}}\max\overline{u}_{0}^{-1}\left(\gamma e^{-(\Lambda+\varepsilon)t}\right)$,
there holds~:
\begin{align*}
\int_{\mathbb{R}}n(t,x,y)dy & \leq\delta+\eta\sqrt{1+B^{2}}\int_{-\zeta}^{+\infty}\Gamma_{0}(s)ds\\
 & \leq\delta+\eta\sqrt{1+B^{2}}\int_{\mathbb{R}}\Gamma_{0}(s)ds,
\end{align*}
which, combined with \eqref{eq:___<_mu}, proves \eqref{eq:___intn_sup}.
\end{proof}

\subsection{The lower bound \eqref{eq:___intn_inf}\label{subsec:Lower_bound}}

This subsection is devoted to the proof of \eqref{eq:___intn_inf}. 
\begin{lem}
\label{lem:gw_borne_inf}Let $\underline{u}_{0}$ satisfy $(Q)$,
see Definition \ref{def:condition_Q}. Then there exists a function
$\underline{\underline{u}}_{0}\colon\mathbb{R}\rightarrow\mathbb{R}$
such that
\begin{itemize}
\item $\underline{\underline{u}}_{0}\leq\underline{u}_{0}$,
\item there exists $\xi_{2}\in\mathbb{R}$ such that $\underline{\underline{u}}_{0}=\underline{u}_{0}$
on $[\xi_{2},+\infty)$,
\item $\underline{\underline{u}}_{0}$ satisfies $(Q)$,
\item $\underline{\underline{u}}_{0}$ is of class $C^{2}$ on $\mathbb{R}$
and there exists $K\geq0$ such that $|\underline{\underline{u}}_{0}^{\prime\prime}|\leq K\underline{\underline{u}}_{0}$
on $\mathbb{R}$.
\end{itemize}
\end{lem}

\begin{proof}
[Proof of Lemma \ref{lem:gw_borne_inf}]From Proposition \ref{prop:ppties_Q},
there exists $\xi_{1}>\xi_{0}$ so that $\underline{u}_{0}(X)\geq\underline{u}_{0}(\xi_{1})$
for any $X\leq\xi_{1}$. Fix $h\in(0,\xi_{1}-\xi_{0})$.

Since $\underline{u}_{0}$ satisfies $(Q)$, there is $\xi_{2}>\xi_{1}$
such that $\underline{u}_{0}(\xi_{2})<\underline{u}_{0}(\xi_{1})$
and $|\underline{u}_{0}^{\prime\prime}(X)|\leq\underline{u}_{0}(X)$
for all $X\geq\xi_{2}$. Next, one can construct a nondecreasing,
concave function $\phi\colon\mathbb{R}_{+}\rightarrow\mathbb{R}_{+}$
of class $C^{2}$ satisfying

\[
\phi(x)=\begin{cases}
x, & \forall x\leq\underline{u}_{0}(\xi_{2}),\\
m, & \forall x\geq\underline{u}_{0}(\xi_{1}),
\end{cases}
\]
for some $m<\underline{u}_{0}(\xi_{1})$. Finally, set $\underline{\underline{u}}_{0}=\phi\circ\underline{u}_{0}$.
Let us prove that $\underline{\underline{u}}_{0}$ satisfies all the
desired properties.

Given that $\phi$ is concave, we have $\phi(x)\leq x$ on $\mathbb{R}_{+}$,
thus $\underline{\underline{u}}_{0}\leq\underline{u}_{0}$ on $\mathbb{R}$.
Meanwhile, for any $X\geq\xi_{2}$, there holds $\underline{u}_{0}(X)\leq\underline{u}_{0}(\xi_{2})$,
which implies $\underline{\underline{u}}_{0}(X)=\underline{u}_{0}(X)$.
Furthermore, the function $\underline{\underline{u}}_{0}$ is of class
$C^{2}$ on $\mathbb{R}$. Indeed, by composition $\underline{\underline{u}}_{0}$
is $C^{2}$ on $(\xi_{1}-h,+\infty)$, whereas on $(-\infty,\xi_{1}]$,
$\underline{\underline{u}}_{0}$ is constant from our choice of $\xi_{1}$.

Let us check that the function $\underline{\underline{u}}_{0}$, which
is clearly bounded, satisfies condition $(Q)$, see Definition \ref{def:condition_Q}.
Given that $\phi$ and $\underline{u}_{0}$ are uniformly continuous,
so is $\underline{\underline{u}}_{0}$. Then, since $\underline{\underline{u}}_{0}=\underline{u}_{0}$
on $[\xi_{2},+\infty)$ and $\underline{u}_{0}$ satisfies $(Q)$,
we collect for free the properties corresponding to the third and
fourth lines of condition $(Q)$, as well as $\lim_{+\infty}\underline{\underline{u}}_{0}=0$.
Meanwhile, $\phi>0$ leads to $\underline{\underline{u}}_{0}>0$.
Eventually, given our choice of $\xi_{1}$, one deduces $\liminf_{-\infty}\underline{u}_{0}\geq\underline{u}_{0}(\xi_{1})$,
whence $\liminf_{-\infty}\underline{\underline{u}}_{0}=m>0$. As a
result, $\underline{\underline{u}}_{0}$ does satisfy $(Q)$.

It remains to prove the existence of a real $K\geq0$ such that $|\underline{\underline{u}}_{0}^{\prime\prime}|\leq K\underline{\underline{u}}_{0}$
on $\mathbb{R}$. From our choice of $\xi_{2}$, there holds
\[
|\underline{\underline{u}}_{0}^{\prime\prime}(X)|\begin{cases}
=0<\underline{\underline{u}}_{0}(X) & \forall X\leq\xi_{1},\\
=|\underline{u}_{0}^{\prime\prime}(X)|\leq\underline{u}_{0}(X)=\underline{\underline{u}}_{0}(X) & \forall X\geq\xi_{2}.
\end{cases}
\]
Meanwhile, $\underline{\underline{u}}_{0}$ is $C^{2}$ and positive
on $[\xi_{1},\xi_{2}]$, thus there exists $K^{\prime}\geq0$ such
that for any $X$ in $[\xi_{1},\xi_{2}]$
\[
|\underline{\underline{u}}_{0}^{\prime\prime}(X)|\leq K^{\prime}\min_{[\xi_{1},\xi_{2}]}\underline{\underline{u}}_{0}\leq K^{\prime}\underline{\underline{u}}_{0}(X).
\]
Thus it suffices to choose $K=\max(1,K^{\prime})$.
\end{proof}
We now turn to the derivation of estimate \eqref{eq:___intn_inf}.
\begin{proof}
[Proof of \eqref{eq:___intn_inf}]The proof involves three steps.
First we construct a subsolution $\underline{v}$ on $[0,1]\times \mathbb{R}\times [\sigma_- -\alpha,\sigma_+ +\alpha]$, which, for $\alpha>0$ large enough, provides a lower bound of the form $v(1,X,Y)\geq \rho \,\underline{\underline{u}}_{0}(X)\Gamma_{0}^{R}(Y)$
for some $\rho>0$, where $\underline{\underline{u}}_{0}$
is constructed from $\underline{u}_{0}$ as in Lemma \ref{lem:gw_borne_inf}.

Next, equipped with this lower bound provided by $\underline{v}$,
we construct a second subsolution $\underline{w}$ on $[1,+\infty)\times\mathbb{R}^{2}$,
which spreads and accelerates in the direction $X\rightarrow+\infty$.
Let us recall that due to the nonlocal term, equation \eqref{eq:dim_2_xy}
does not satisfy the comparison principle. Thus, to prove that $v(t,\cdot,\cdot)\geq\underline{w}(t,\cdot,\cdot)$
for $t\geq1$, we invoke a refinement of the parabolic Harnack inequality,
that is Theorem \ref{thm:Harnack_gen}.

The last step consists in transferring the estimates on the level
sets of $\underline{w}$ into estimates on $E_{\mu}^{n}(t)$, similarly
to the proof of \eqref{eq:___intn_sup} in subsection \ref{subsec:Upper_bound}.

\paragraph{First subsolution.}Let $\alpha>0$ large enough so that
$[-R,R]\subset(\sigma_{-}-\alpha,\sigma_{+}+\alpha)$. Let $p(t,Y)$
be the solution of the initial boundary value problem
\[
\begin{cases}
\partial_{t}p-\partial_{YY}p=0, & t\in(0,1],\text{ }Y\in(\sigma_{-}-\alpha ,\sigma_{+}+\alpha ),\\
p(t,Y)=0, & t\in(0,1],\text{ }Y=\sigma_{\pm}\pm\alpha,\\
p(0,Y)=p_{0}(Y), & Y\in(\sigma_{-},\sigma_{+}),
\end{cases}
\]
with $p_{0}$ the quadratic polynomial that satisfies $p_{0}(\sigma_{\pm})=0$
and $p_{0}\left(\frac{\sigma_{+}+\sigma_{-}}{2}\right)=1$. In other
words, $p(t,Y)$ solves the one-dimensional heat equation on $[0,1]\times [\sigma_- -\alpha, \sigma_+ + \alpha]$, with zero Dirichlet
conditions imposed on the boundary. Let us recall that $N_{\infty}$,
$C$ and $\kappa$ are positive real numbers such that \eqref{eq:maj_intn}
and \eqref{eq:maj_n} hold. Moreover, since $r\in L_{loc}^{\infty}(\mathbb{R})$,
there exists $r_{min}\leq0$ such that $r(Y)\geq r_{min}$ for any
$Y\in[\sigma_{-}-\alpha,\sigma_{+}+\alpha]$. We define the following
subsolution
\begin{equation}
\underline{v}(t,X,Y)\coloneqq e^{-kt}\underline{\underline{u}}_{0}(X)p(t,Y),\label{eq:___subsol_1}
\end{equation}
\begin{equation}
k\coloneqq K-r_{min}+N_{\infty},
\end{equation}
\begin{equation}
\Omega\coloneqq\left\{ (t,X,Y)\mid0<t<1,\text{ }X\in\mathbb{R},\text{ }Y\in(\sigma_{-}-\alpha,\sigma_{+}+\alpha)\right\} ,\label{eq:___subsol_1_dom}
\end{equation}
where $\underline{\underline{u}}_{0}$ is constructed from $\underline{u}_{0}$
as in Lemma \ref{lem:gw_borne_inf}, with the associated constant
$K\geq0$, so that $k\geq N_{\infty}>0$. 

Let us prove that $v\geq\underline{v}$ on $\overline{\Omega}$. We
first check that $v\leq\underline{v}$ on the parabolic boundary of
$\Omega$, that is
\[
\partial_{p}\Omega=\Bigl(\{0\}\times\mathbb{R}\times[\sigma_{-},\sigma_{+}]\Bigr)\bigcup D_{+}\bigcup D_{-},
\]
where 
\[
D_{\pm}\coloneqq\bigl\{(t,X,\sigma_{\pm}\pm\alpha )\mid t\in(0,1],X\in\mathbb{R}\bigr\}.
\]
On the one hand, it follows from \eqref{eq:encad_v0} that $v_{0}(X,Y)\geq\underline{v}(0,X,Y)$
for $X\in\mathbb{R}$ and $Y\in[\sigma_{-},\sigma_{+}]$. On the other
hand, on $D_{+}\cup D_{-}$, one has $\underline{v}=0\leq v$. Thus
$v\geq\underline{v}$ on $\partial_{p}\Omega$. It remains to show
that $v-\underline{v}$ is a supersolution of a parabolic problem
on $\Omega$. First, for $t\in(0,1)$, there holds
\begin{align*}
\partial_{t}v-\partial_{XX}v-\partial_{YY}v & =r(Y)v-v\int_{\mathbb{R}}v(t,\chi,\psi)dy\\
 & \geq(r_{min}-N_{\infty})v.
\end{align*}
Meanwhile, since $|\underline{\underline{u}}_{0}^{\prime\prime}|\leq K\underline{\underline{u}}_{0}$,
one obtains 
\begin{align*}
\partial_{t}\underline{v}-\partial_{XX}\underline{v}-\partial_{YY}\underline{v} & =e^{-kt}\left[-k\underline{\underline{u}}_{0}p-\underline{\underline{u}}_{0}^{\prime\prime}p\right]\\
 & =e^{-kt}\left[(r_{min}-N_{\infty})\underline{\underline{u}}_{0}p+(-K\underline{\underline{u}}_{0}-\underline{\underline{u}}_{0}^{\prime\prime})p\right]\\
 & \leq(r_{min}-N_{\infty})\underline{v}.
\end{align*}
We conclude by the maximum principle that $v\geq\underline{v}$ on
$\overline{\Omega}$. We have thus completed the first step of the
proof.

\paragraph{Second subsolution.}We now turn to the construction of
a second subsolution. The maximum principle shows that $p(1,Y)$ is
positive on $(\sigma_{-}-\alpha,\sigma_{+}+\alpha)$. Since $[-R,R]\subset(\sigma_{-}-\alpha,\sigma_{+}+\alpha)$, there
exists $p_{min}>0$ so that $p(1,Y)\geq p_{min}$ for any $Y\in[-R,R]$.
Thus, for any $X\in\mathbb{R}$ and $Y\in[-R,R]$, there holds
\begin{align*}
v(1,X,Y) & \geq\underline{v}(1,X,Y)\\
 & \geq e^{-k}p_{min}\underline{\underline{u}}_{0}(X)\\
 & \geq e^{-k}p_{min}\underline{\underline{u}}_{0}(X)\Gamma_{0}^{R}(Y),
\end{align*}
where $\Gamma_{0}^{R}$ solves \eqref{eq:Eigen_R}. Fix now any real
number $\rho>0$ small enough so that $\rho<||\underline{\underline{u}}_{0}||_{\infty}e^{-k}p_{min}$.
We allow ourselves to take $\rho$ even smaller if needed. Set $\Lambda_{\varepsilon}^{R}=-\lambda_{0}^{R}-\varepsilon/2>0$.
For any $t\geq1$ and $X,Y\in\mathbb{R}$, define
\[
\underline{w}(t,X,Y)\coloneqq\rho u(t,X)\Gamma_{0}^{R}(Y),
\]
where $u(t,X)$ solves the Fisher-KPP equation
\[
\begin{cases}
u_{t}-u_{XX}=\Lambda_{\varepsilon}^{R}u(1-u), & t>1,\text{ }X\in\mathbb{R},\\
u(1,X)=\underline{\underline{u}}_{0}(X)/||\underline{\underline{u}}_{0}||_{\infty}, & X\in\mathbb{R}.
\end{cases}
\]
In particular, $u(t,X)\in[0,1]$ thanks to the maximum principle.

We shall now prove that $v(t,X,Y)>\underline{w}(t,X,Y)$ on $[1,+\infty)\times\mathbb{R}^{2}$.
Given our choice of $\rho$, we indeed have $v(1,X,Y)>\underline{w}(1,X,Y)$
on $\mathbb{R}^{2}$. Assume by contradiction that the closed set
\[
E=\bigl\{ t>1\mid\exists(X,Y)\in\mathbb{R}^{2},\quad v(t,X,Y)=\underline{w}(t,X,Y)\bigr\},
\]
is nonempty. Set $t_{0}\coloneqq\min E>1$ and $(X_{0},Y_{0})\in\mathbb{R}^{2}$
the point where $v(t_{0},X_{0},Y_{0})=\underline{w}(t_{0},X_{0},Y_{0})$.
Note that this implies $Y_{0}\in(-R,R)$, since the maximum principle
yields $v(t,\cdot,\cdot)>0$ for any $t>0$. Before going further,
we first use Theorem \ref{thm:Harnack_gen} to estimate the nonlocal
term in \eqref{eq:dim_2_XY}. Let $(x_{0},y_{0})$ be the corresponding
coordinates of $(X_{0},Y_{0})$ obtained through the change of variable
\eqref{eq:XY_eq_xy}. Fix $M>0$, large enough so that $\frac{3C}{\kappa}e^{-\kappa M}\leq\varepsilon/8$
and $|y_{0}-Bx_{0}|=\sqrt{1+B^{2}}|Y_{0}|\leq M$. Thanks to the control
of the tails \eqref{eq:maj_n}, there holds
\begin{align}
\int_{\mathbb{R}}v(t_{0},\chi(X_{0},Y_{0},y),\psi(X_{0},Y_{0},y))dy & =\int_{\mathbb{R}}v\left(t_{0},\frac{\frac{X_{0}-BY_{0}}{\sqrt{1+B^{2}}}+By}{\sqrt{1+B^{2}}},\frac{-B\frac{X_{0}-BY_{0}}{\sqrt{1+B^{2}}}+y}{\sqrt{1+B^{2}}}\right)dy\nonumber \\
 & =\int_{\mathbb{R}}n(t_{0},x_{0},y)dy=\int_{\mathbb{R}}n(t_{0},x_{0},Bx_{0}+y)dy\nonumber \\
 & \leq2M\max_{y\in[-M,M]}n(t_{0},x_{0},Bx_{0}+y)+\int_{[-M,M]^{c}}Ce^{-\kappa|y|}dy.\label{eq:___maj_intv}
\end{align}

Next, in order to estimate the first term of \eqref{eq:___maj_intv},
let us recall that the solutions are uniformly bounded, as implied
by \eqref{eq:maj_n}. This allows us to use the refinement of the
Harnack inequality, namely Theorem \ref{thm:Harnack_gen}, with $\delta=\frac{C}{2M\kappa}e^{-\kappa M}>0$.
Thus there exists a constant $C_{M}>0$ such that
\begin{align*}
\max_{(x,y)\in[-M,M]^{2}}n(t_{0},x_{0}+x,Bx_{0}+y) & \leq C_{M}\min_{(x,y)\in[-M,M]^{2}}n(t_{0},x_{0}+x,Bx_{0}+y)+\delta\\
 & \leq C_{M}n(t_{0},x_{0},y_{0})+\delta,
\end{align*}
which we plug into \eqref{eq:___maj_intv} to obtain
\[
\int_{\mathbb{R}}v\left(t_{0},\frac{\frac{X_{0}-BY_{0}}{\sqrt{1+B^{2}}}+By^{\prime}}{\sqrt{1+B^{2}}},\frac{-B\frac{X_{0}-BY_{0}}{\sqrt{1+B^{2}}}+y^{\prime}}{\sqrt{1+B^{2}}}\right)dy^{\prime}\leq2MC_{M}v(t_{0},X_{0},Y_{0})+\frac{3C}{\kappa}e^{-\kappa M}.
\]

Going back to our proof by contradiction, since $(\underline{w}-v)$
is negative on $[1,t_{0})\times\mathbb{R}^{2}$, it reaches its maximum
on $[1,t_{0}]\times\mathbb{R}^{2}$ at the point $(t_{0},X_{0},Y_{0})$.
Thus
\begin{equation}
\bigl[\partial_{t}(\underline{w}-v)-\partial_{XX}(\underline{w}-v)-\partial_{YY}(\underline{w}-v)-r(Y_{0})(\underline{w}-v)\bigr](t_{0},X_{0},Y_{0})\geq0,\label{eq:___ineg_absurde}
\end{equation}
On the one hand, there holds
\begin{align*}
-\left[\partial_{t}v-\partial_{XX}v-\partial_{YY}v-r(Y)v\right](t_{0},X_{0},Y_{0}) & =v(t_{0},X_{0},Y_{0})\int_{\mathbb{R}}v(t_{0},\chi,\psi)dy\\
 & \leq2MC_{M}\underline{w}(t_{0},X_{0},Y_{0})^{2}+\frac{3C}{\kappa}e^{-\kappa M}\underline{w}(t_{0},X_{0},Y_{0}).
\end{align*}
On the other hand, one has
\begin{align*}
\partial_{t}\underline{w}-\partial_{XX}\underline{w}-\partial_{YY}\underline{w}-r(Y)\underline{w} & =\rho\left(\partial_{t}u-\partial_{XX}u+\lambda_{0}^{R}u\right)\Gamma_{0}^{R}\\
 & \leq\rho\left(\Lambda_{\varepsilon}^{R}u+\lambda_{0}^{R}u\right)\Gamma_{0}^{R}\\
 & \leq-\frac{\varepsilon}{2}\underline{w}.
\end{align*}
Thus, \eqref{eq:___ineg_absurde} leads to
\begin{align*}
0 & \leq-\frac{\varepsilon}{2}\underline{w}(t_{0},X_{0},Y_{0})+2MC_{M}\underline{w}(t_{0},X_{0},Y_{0})^{2}+\frac{3C}{\kappa}e^{-\kappa M}\underline{w}(t_{0},X_{0},Y_{0})\\
 & \leq\left[-\frac{\varepsilon}{2}+2MC_{M}\rho+\frac{3C}{\kappa}e^{-\kappa M}\right]\underline{w}(t_{0},X_{0},Y_{0})\\
 & \leq\left[-\frac{\varepsilon}{2}+\frac{\varepsilon}{8}+\frac{\varepsilon}{8}\right]\underline{w}(t_{0},X_{0},Y_{0}),
\end{align*}
provided we select $\rho$ small enough so that $2MC_{M}\rho\leq\varepsilon/8$.
Note that reducing $\rho$ may change the values of $t_{0},X_{0},Y_{0}$
but all above estimates remain true since there always holds $Y_{0}\in(-R,R)$
and $C_{M}$ does not depend on $t_{0}\geq1$. In the end, one has
$-\varepsilon\underline{w}(t_{0},X_{0},Y_{0})/4\geq0$. This implies
$\underline{w}(t_{0},X_{0},Y_{0})\leq0$, which is absurd. As a result,
$\underline{w}(t,X,Y)<v(t,X,Y)$ for any $t\geq1$ and $X,Y\in\mathbb{R}$.

\paragraph{Conclusion.}Before going further, let us mention that,
for any $t>1$ and $X\in\mathbb{R}$, there holds $u(t,X)<1$, thus
\begin{align*}
\int_{\mathbb{R}}w(t,X,Y)dy & <\rho\int_{\mathbb{R}}\Gamma_{0}^{R}(Y)dy=\rho\sqrt{1+B^{2}}\int_{\mathbb{R}}\Gamma_{0}^{R}(y)dy\eqqcolon\beta.
\end{align*}
Therefore, the lower bound $v(t,X,Y)\geq\underline{w}(t,X,Y)$ on
$[1,+\infty)\times\mathbb{R}^{2}$ does not provide any information
on the location of $E_{\mu}^{n}(t)$ for levels $\mu\geq\beta$. As
a result the location of larger levels are seemingly out of reach,
which is typical of equations without comparison principle, as already
mentioned in the introduction and subsection \ref{subsec:Main_results}.

Now, given any $\mu\in(0,\beta)$, one can select $\eta=\eta(\mu)\in(0,1)$
such that $\eta>\mu/\beta$. Since $v\geq\underline{w}$ on $[1,+\infty)\times\mathbb{R}^{2}$,
there holds for any $t\geq1$ 
\begin{align*}
\int_{\mathbb{R}}n(t,x,y)dy & \geq\int_{\mathbb{R}}\underline{w}(t,X,Y)dy\\
 & \geq\int_{\mathbb{R}}\rho u(t,X)\Gamma_{0}^{R}(Y)dy\\
 & \geq\rho\int_{\mathbb{R}}u\left(t,\frac{x+By}{\sqrt{1+B^{2}}}\right)\Gamma_{0}^{R}\left(\frac{y-Bx}{\sqrt{1+B^{2}}}\right)dy\\
 & \geq\rho\sqrt{1+B^{2}}\int_{\mathbb{R}}u(t,\sqrt{1+B^{2}}x+Bs)\Gamma_{0}^{R}(s)ds\\
 & \geq\rho\sqrt{1+B^{2}}\int_{-R}^{R}u(t,\sqrt{1+B^{2}}x+Bs)\Gamma_{0}^{R}(s)ds.
\end{align*}

We will show that for any
\[
x\leq\frac{1}{\sqrt{1+B^{2}}}\min\underline{u}_{0}^{-1}\left(\Gamma e^{-(\Lambda_{\varepsilon}^{R}-\varepsilon/2)t}\right),
\]
and for $t$ large enough there holds $u(t,\sqrt{1+B^{2}}x+Bs)\geq\eta$
for $|s|\leq R$. For now, let us only assume the condition on $x$.
By applying Proposition \ref{prop:ppties_Q} to $\underline{u}_{0}$
with
\[
\begin{cases}
a=\Lambda_{\varepsilon}^{R}-\varepsilon/2, & b=\Lambda_{\varepsilon}^{R}-\varepsilon/4,\\
\Gamma_{a}=\Gamma, & \Gamma_{b}=\Gamma,\\
\chi=BR,
\end{cases}
\]
we deduce that there exists $t_{a,b,\Gamma_{a},\Gamma_{b},\chi}^{\ast}=t_{\Gamma,R,\varepsilon}^{\ast}\geq0$
such that for any $s\in[-R,R]$ and $t\geq t_{\Gamma,R,\varepsilon}^{\ast}$,
there holds
\begin{align*}
\sqrt{1+B^{2}}x+Bs & \leq\min\underline{u}_{0}^{-1}\left(\Gamma e^{-(\Lambda_{\varepsilon}^{R}-\varepsilon/2)t}\right)+Bs\\
 & \leq\min\underline{u}_{0}^{-1}\left(\Gamma e^{-(\Lambda_{\varepsilon}^{R}-\varepsilon/2)t}\right)+BR\\
 & \leq\min\underline{u}_{0}^{-1}\left(\Gamma e^{-(\Lambda_{\varepsilon}^{R}-\varepsilon/4)t}\right).
\end{align*}
Since $\eta\in(0,1)$, Lemma \ref{lem:HR_gen} gives the existence
of a real $T_{\eta,\varepsilon/4,\Gamma,\Lambda_{\varepsilon}^{R}}\geq0$
such that 
\[
\min E_{\eta}(t)\geq\min\underline{\underline{u}}_{0}^{-1}\left(\Gamma e^{-(\Lambda_{\varepsilon}^{R}-\varepsilon/4)t}\right),\qquad\forall t\geq T_{\eta,\varepsilon/4,\Gamma,\Lambda_{\varepsilon}^{R}},
\]
where $E_{\eta}(t)=\{x\in\mathbb{R}\mid u(t,x)=\eta\}$. However,
Lemma \ref{lem:gw_borne_inf} provides some $\xi_{2}\in\mathbb{R}$
such that $\underline{\underline{u}}_{0}(X)=\underline{u}_{0}(X)$
for $X\geq\xi_{2}$, and since $\underline{\underline{u}}_{0}$ satisfies
$(Q)$ (see Definition \ref{def:condition_Q}), one can take $T_{\eta,\varepsilon/4,\Gamma,\Lambda_{\varepsilon}^{R}}$
possibly even larger so that $\min\underline{\underline{u}}_{0}^{-1}\left(\Gamma e^{-(\Lambda_{\varepsilon}^{R}-\varepsilon/4)t}\right)\geq\xi_{2}$.
Also, since $\underline{\underline{u}}_{0}\leq\underline{u}_{0}$,
there holds $\min\underline{u}_{0}^{-1}\left(\Gamma e^{-(\Lambda_{\varepsilon}^{R}-\varepsilon/4)t}\right)\geq\xi_{2}$,
so that
\[
\min\underline{\underline{u}}_{0}^{-1}\left(\Gamma e^{-(\Lambda_{\varepsilon}^{R}-\varepsilon/4)t}\right)=\min\underline{u}_{0}^{-1}\left(\Gamma e^{-(\Lambda_{\varepsilon}^{R}-\varepsilon/4)t}\right),\qquad\forall t\geq T_{\eta,\varepsilon/4,\Gamma,\Lambda_{\varepsilon}^{R}}.
\]
Additionally, seeing that $\inf_{X\leq0}u(t,X)\rightarrow1$ as $t\rightarrow+\infty$,
there is $\boldsymbol{t}_{\eta,R,\varepsilon}\geq0$ such that 
\[
\liminf_{X\rightarrow-\infty}u(t,X)>\eta,\qquad\forall t\geq\boldsymbol{t}_{\eta,R,\varepsilon}.
\]
Finally, set
\begin{align*}
T_{\eta,\varepsilon,\Gamma,R}^{\ast} & \coloneqq\max(t_{\Gamma,R,\varepsilon}^{\ast},\,T_{\eta,\varepsilon/4,\Gamma,\Lambda_{\varepsilon}^{R}},\,\boldsymbol{t}_{\eta,R,\varepsilon}),
\end{align*}
then for any $x\leq\frac{1}{\sqrt{1+B^{2}}}\min\underline{u}_{0}^{-1}\left(\Gamma e^{-(\Lambda_{\varepsilon}^{R}-\varepsilon/2)t}\right)$
and $t\geq T_{\eta,\varepsilon,\Gamma,R}^{\ast}$, there holds
\begin{align}
\min E_{\eta}(t) & \geq\min\underline{u}_{0}^{-1}\left(\Gamma e^{-(\Lambda_{\varepsilon}^{R}-\varepsilon/4)t}\right)\nonumber \\
 & \geq\sqrt{1+B^{2}}x+BR.\label{eq:___ls_min}
\end{align}
Thus for any $s\in(-R,R)$, one has $u(t,\sqrt{1+B^{2}}x+Bs)>\eta$.
Indeed, assume by contradiction that there exists $s_{0}\in(-R,R)$
such that $u(t,\sqrt{1+B^{2}}x+Bs_{0})\leq\eta$. Given that $t\geq\boldsymbol{t}_{\eta,R,\varepsilon}$,
there would exist $x_{1}\in E_{\eta}(t)\cap(-\infty,\sqrt{1+B^{2}}x+Bs_{0})$,
which is absurd considering \eqref{eq:___ls_min}.

Consequently, whenever $t\geq T_{\eta,\varepsilon,\Gamma,R}^{\ast}$
and $x\leq\frac{1}{\sqrt{1+B^{2}}}\min\underline{u}_{0}^{-1}\left(\Gamma e^{-(\Lambda_{\varepsilon}^{R}-\varepsilon/2)t}\right)$,
there holds 
\begin{align*}
\int_{\mathbb{R}}n(t,x,y)dy & \geq\rho\sqrt{1+B^{2}}\int_{-R}^{R}\eta\Gamma_{0}^{R}(s)ds\\
 & \geq\eta\beta>\mu.
\end{align*}
Since $\Lambda_{\varepsilon}^{R}-\varepsilon/2=-\lambda_{0}^{R}-\varepsilon$,
this concludes the proof of \eqref{eq:___intn_inf}.
\end{proof}

\subsection{Conclusion\label{subsec:Proof_conc}}
\begin{proof}
[Proof of Theorem \ref{thm:gw_main}]Let us recall that \eqref{eq:___intn_sup}-\eqref{eq:___intn_inf}
have been established in subsections \ref{subsec:Upper_bound} and
\ref{subsec:Lower_bound} respectively. We can now complete the proof
of Theorem \ref{thm:gw_main}. Set
\[
N_{t}(x)\coloneqq\int_{\mathbb{R}}n(t,x,y)dy,
\]
for any $t\geq0$ and $x\in\mathbb{R}$. Set $T_{\mu,\varepsilon,\gamma,\Gamma,R}^{\ast}\coloneqq\max(T_{\mu,\varepsilon,\gamma}^{\ast},T_{\mu,\varepsilon,\Gamma,R}^{\ast})$.
One can reformulate \eqref{eq:___intn_sup}-\eqref{eq:___intn_inf}
as follows~: there exists $\beta>0$ so that for any $\mu\in(0,\beta)$,
$\varepsilon\in(0,-\lambda_{0}^{R})$, $\Gamma>0$ and $\gamma>0$,
there exists $T_{\mu,\varepsilon,\gamma,\Gamma,R}^{\ast}\geq0$ such
that for all $t\geq T_{\mu,\varepsilon,\gamma,\Gamma,R}^{\ast}$,
there holds

\[
N_{t}(x)<\mu,\qquad\forall x\geq\frac{1}{\sqrt{1+B^{2}}}\max\overline{u}_{0}^{-1}\left(\gamma e^{-(-\lambda_{0}+\varepsilon)t}\right),
\]
\[
N_{t}(x)>\mu,\qquad\forall x\leq\frac{1}{\sqrt{1+B^{2}}}\min\underline{u}_{0}^{-1}\left(\Gamma e^{-(-\lambda_{0}^{R}-\varepsilon)t}\right),
\]
Since $N_{t}$ is continuous for any $t>0$, and $E_{\mu}^{n}(t)=N_{t}^{-1}(\mu)$,
the level set $E_{\mu}^{n}(t)$ is nonempty, closed, and included
in 
\[
\frac{1}{\sqrt{1+B^{2}}}\left[\min\underline{u}_{0}^{-1}\left(\Gamma e^{-(-\lambda_{0}^{R}-\varepsilon)t}\right),\max\overline{u}_{0}^{-1}\left(\gamma e^{-(-\lambda_{0}+\varepsilon)t}\right)\right],
\]
for any $t\geq T_{\mu,\varepsilon,\gamma,\Gamma,R}^{\ast}$. This
concludes the proof of Theorem \ref{thm:gw_main}.
\end{proof}

\subsection{A heavy tail induces acceleration\label{subsec:HT_accel_proof}}

In this short subsection, we give a sketch of the proof of Theorem
\ref{thm:gw_HT_accel}, as it follows the same lines as the proof
of \eqref{eq:___intn_inf} done in subsection \ref{subsec:Lower_bound}. 

By assumption, there holds $\lambda_{0}<0$. Select $R>0$ large enough
so that $\lambda_{0}^{R}<0$ (see subsection \ref{subsec:Eigen_intro}).
Choose any $c>2\sqrt{-\lambda_{0}^{R}}$ and set $\alpha_{c}\in(0,\sqrt{-\lambda_{0}^{R}})$
the only real satisfying $c=\alpha_{c}+\frac{-\lambda_{0}^{R}}{\alpha_{c}}$.
Since there exists $m>0$ such that $\underline{u}_{0}(X)\geq\min(m,e^{-\alpha_{c}X})$,
we can construct, as in the proof of Lemma \ref{lem:gw_borne_inf},
a $C^{2}$ function $\underline{\underline{u}}_{0}\colon\mathbb{R}\rightarrow[0,1]$
that satisfies 
\begin{itemize}
\item $\underline{\underline{u}}_{0}\leq\underline{u}_{0}$,
\item $\underline{\underline{u}}_{0}$ is asymptotically front-like, i.e.
satisfies \eqref{eq:asymp_frontL},
\item $\underline{\underline{u}}_{0}(X)=e^{-\alpha_{c}X}$ for $X$ large
enough,
\item there exists $K\geq0$ such that $|\underline{\underline{u}}_{0}^{\prime\prime}|\leq K\underline{\underline{u}}_{0}$
on $\mathbb{R}$.
\end{itemize}
Then, select $\alpha>0$ large enough so that $[-R,R]\subset(\sigma_{-}-\alpha,\sigma_{+}+\alpha)$.
Set $\underline{v}$ as in \eqref{eq:___subsol_1}-\eqref{eq:___subsol_1_dom}.
From \eqref{eq:2D_heavy_tail}, the maximum principle allows us to
conclude that $v\geq\underline{v}$ on $[0,1]\times\mathbb{R}^{2}$.

Next, as in subsection \ref{subsec:Lower_bound}, we can construct
a second subsolution on $[1,+\infty)\times\mathbb{R}^{2}$ of the
form $\underline{w}(t,X,Y)=\rho u(t,X)\Gamma_{0}^{R}(Y)$ for some
$\rho>0$, where $\Gamma_{0}^{R}$ solves \eqref{eq:Eigen_R} and
$u(t,X)$ solves 
\[
\begin{cases}
u_{t}-u_{XX}=-\lambda_{0}^{R}u(1-u), & t>1,\text{ }X\in\mathbb{R},\\
u(1,X)=\underline{\underline{u}}_{0}(X), & X\in\mathbb{R}.
\end{cases}
\]
However, since $\underline{\underline{u}}_{0}(X)$ decays as $e^{-\alpha_{c}X}$,
the function $u(t,X)$ converges to a shift of the front $\varphi_{c}(X-ct)$
solution of the Fisher-KPP equation $u_{t}-u_{XX}=-\lambda_{0}^{R}u(1-u)$,
see \cite{Uch_78}. As a consequence, with the same calculation as
in subsection \ref{subsec:Lower_bound}, we deduce the existence of
a level $\beta>0$ such that for any $\mu\in(0,\beta)$, there holds
\[
\liminf_{t\rightarrow+\infty}\left(\frac{1}{t}\min E_{\mu}^{n}(t)\right)\geq\frac{1}{\sqrt{1+B^{2}}}c,
\]
Since $c>2\sqrt{-\lambda_{0}^{R}}$ may be chosen arbitrarily large,
there holds $\min E_{\mu}^{n}(t)/t\rightarrow+\infty$ as $t\rightarrow+\infty$,
for all $\mu\in(0,\beta)$, leading to the result of Theorem \ref{thm:gw_HT_accel}.

\section{No acceleration for ill-directed heavy tails\label{sec:HTx}}

In contradistinction with Section \ref{sec:Main_proof}, we will keep
the notation $\tilde{r}(Y)=r(y-Bx)=r(\sqrt{1+B^{2}}Y)$, since the
two coordinate systems $(x,y)$ and $(X,Y)$ will be used in conjunction
during the proof.
\begin{proof}
[Proof of Theorem \ref{thm:gw_HTx}]We only give the proof for $c>c^{\ast}$,
as the case $c<-c^{\ast}$ is similar. We first set the following
positive constants
\[
s^{\ast}\coloneqq c^{\ast}\sqrt{1+B^{2}}=2\sqrt{-\lambda_{0}},\qquad\gamma\coloneqq\sqrt{-\lambda_{0}}.
\]
Fix any $s>s^{\ast}$, and define
\begin{align*}
\psi(t,x,y) & \coloneqq Ce^{-\gamma(X-st)}\Gamma_{0}(Y),\\
\varphi(t,x,y) & \coloneqq e^{-\alpha t}u(t,x)p(t,y),
\end{align*}
where $X,Y$ are given by \eqref{eq:XY_eq_xy}, $\alpha\geq\gamma^{2}\left(1+\frac{1}{B^{2}}\right)$
may be chosen arbitrarily large, $C>0$ is a positive constant to
be determined later, and the functions $u(t,x)$, $p(t,y)$ respectively
solve 
\begin{equation}
\begin{cases}
\partial_{t}u-\partial_{xx}u=u(||u_{0}||_{\infty}-u) & t>0,\,x\in\mathbb{R},\\
u(0,x)=u_{0}(x) & x\in\mathbb{R},
\end{cases}\label{eq:___HTx_u_eq}
\end{equation}
\begin{equation}
\begin{cases}
\partial_{t}p-\partial_{yy}p=0 & t>0,\,y\in\mathbb{R},\\
p(0,y)=\boldsymbol{1}_{[\sigma_{-},\sigma_{+}]}(y) & y\in\mathbb{R}.
\end{cases}\label{eq:___HTx_p_eq}
\end{equation}
Since $p$ solves the one-dimensional heat equation, it is expressed
as the convolution
\begin{equation}
p(t,y)=\frac{1}{\sqrt{4\pi t}}\int_{\sigma_{-}}^{\sigma_{+}}e^{-(y-z)^{2}/(4t)}dz.\label{eq:___p_convol}
\end{equation}
We shall prove that $\psi+\varphi\geq n$ using the maximum principle.
Notice that one clearly has $\psi(0,\cdot,\cdot)\geq0$, and $\varphi(0,x,y)\geq n_{0}(x,y)$
with \eqref{eq:borne_n0_HTx}. Also, since nonnegative, $n$ is a
subsolution of the linear local operator~:
\begin{align*}
\mathcal{L} & \coloneqq\partial_{t}-\partial_{xx}-\partial_{yy}-r(y-Bx),\\
 & =\partial_{t}-\partial_{XX}-\partial_{YY}-\tilde{r}(Y),
\end{align*}
thus it suffices to prove that $\mathcal{L}(\psi+\varphi)\geq0$ on
$(0,+\infty)\times\mathbb{R}^{2}$ to conclude. 

Given that $-\partial_{YY}\Gamma_{0}-\tilde{r}(Y)\Gamma_{0}(Y)=\lambda_{0}\Gamma_{0}(Y)$,
we have
\begin{align*}
\mathcal{L}\psi & =\left(\gamma s-\gamma^{2}+\lambda_{0}\right)\psi>0,
\end{align*}
since $\gamma s^{\ast}-\gamma^{2}+\lambda_{0}=0$. In particular,
$\psi$ is a supersolution.

Let us now turn our attention to the function $\varphi$. Since $r$
satisfies Assumption \ref{assu:func_r_k_n0}, there exists $Y_{0}>0$
such that $\tilde{r}(Y)\leq-\alpha$ whenever $|Y|\geq Y_{0}$. Set
\[
\Omega_{0}\coloneqq\left\{ (x,y)\in\mathbb{R}^{2}\;\biggl|\;|Y|=\frac{1}{\sqrt{1+B^{2}}}\left|y-Bx\right|<Y_{0}\right\} .
\]
 On $\Omega_{0}^{c}$, there holds $-\tilde{r}(Y)=-r(y-Bx)\geq\alpha$.
Therefore this implies 
\begin{align*}
\mathcal{L}\varphi & =-\alpha\varphi+(||u_{0}||_{\infty}-u)\varphi-r(y-Bx)\varphi\geq(||u_{0}||_{\infty}-u)\varphi\geq0.
\end{align*}
Thus $\mathcal{L}(\psi+\varphi)\geq0$ on $(0,+\infty)\times\Omega_{0}^{c}$.

Next, let us consider the domain $\Omega_{0}$. On this domain, $\varphi$
may no longer be a supersolution because $r$ may be greater than
$-\alpha$. However we shall prove that $\psi+\varphi$ is a supersolution.
Indeed, there holds

\begin{subequations}
\label{___L_expr}
\begin{align}
\mathcal{L}\psi & \geq Q_{1}e^{-\gamma(X-st)},
\end{align}
\begin{equation}
\mathcal{L}\varphi\geq-Q_{2}e^{-\alpha t}p(t,y),
\end{equation}
\end{subequations}
where

\begin{subequations}
\label{___Q_constants}
\begin{align}
Q_{1} & \coloneqq\left(\gamma s-\gamma^{2}+\lambda_{0}\right)C\min_{|Y|\leq Y_{0}}\Gamma_{0}(Y)>0,
\end{align}
\begin{equation}
Q_{2}\coloneqq(\alpha+r_{max})||u_{0}||_{\infty}.
\end{equation}
\end{subequations}

Now, let us divide $\Omega_{0}$ into two parts~:
\begin{align*}
\Omega_{-} & \coloneqq\Omega_{0}\cap\{(x,y)\mid BX-Y_{0}\leq\theta\},\\
\Omega_{+} & \coloneqq\Omega_{0}\cap\{(x,y)\mid BX-Y_{0}>\theta\},
\end{align*}
where
\[
\theta\coloneqq\max\left(1,(\sigma_{+}+1)\sqrt{1+B^{2}}\right)>0.
\]

On the domain $\Omega_{-}$, one has 
\begin{align*}
\mathcal{L}\psi & \geq Q_{1}e^{-\gamma X}\geq Q_{1}e^{-\gamma(Y_{0}+\theta)/B}\\
\mathcal{L}\varphi & \geq-Q_{2}e^{-\alpha t}||p(0,\cdot)||_{\infty}\geq-Q_{2}.
\end{align*}
In view of \eqref{___Q_constants} it suffices to take $C$ large
enough to reach $\mathcal{L}(\psi+\varphi)\geq0$ on $(0,+\infty)\times\Omega_{-}$.

It remains to prove that $\mathcal{L}(\psi+\varphi)\geq0$ on $(0,+\infty)\times\Omega_{+}$.
For any $(x,y)\in\Omega_{+}$, we have
\[
y=\frac{BX+Y}{\sqrt{1+B^{2}}}\geq\frac{BX-Y_{0}}{\sqrt{1+B^{2}}}\geq\frac{\theta}{\sqrt{1+B^{2}}}>\sigma_{+}.
\]
Consequently, from \eqref{eq:___p_convol} we obtain
\begin{align*}
p(t,y) & \leq\frac{1}{\sqrt{4\pi t}}(\sigma_{+}-\sigma_{-})e^{-(y-\sigma_{+})^{2}/(4t)}\\
 & =\frac{1}{\sqrt{4\pi t}}(\sigma_{+}-\sigma_{-})\exp\left(-\frac{1}{4t}\left(\frac{BX+Y}{\sqrt{1+B^{2}}}-\sigma_{+}\right)^{2}\right)\\
 & \leq\frac{1}{\sqrt{4\pi t}}(\sigma_{+}-\sigma_{-})e^{-Z^{2}/4t},
\end{align*}
where 
\[
Z\coloneqq\frac{BX-Y_{0}}{\sqrt{1+B^{2}}}-\sigma_{+}\geq\frac{\theta}{\sqrt{1+B^{2}}}-\sigma_{+}\geq1.
\]
As a result, from \eqref{___L_expr}, there holds

\begin{align*}
\mathcal{L}(\psi+\varphi) & \geq Q_{1}e^{-\gamma X}e^{\gamma st}-Q_{2}e^{-\alpha t}p(t,y)\\
 & \geq\widetilde{Q_{1}}e^{-\gamma Z\sqrt{1+1/B^{2}}}e^{\gamma st}-\widetilde{Q_{2}}t^{-1/2}e^{-\alpha t}e^{-Z^{2}/4t},
\end{align*}
where 
\begin{align*}
\widetilde{Q_{1}} & \coloneqq Q_{1}e^{-\gamma\left(Y_{0}+\sqrt{1+B^{2}}\sigma_{+}\right)/B},\\
\widetilde{Q_{2}} & \coloneqq Q_{2}(\sigma_{+}-\sigma_{-})\frac{1}{\sqrt{4\pi}}.
\end{align*}
Thus $\mathcal{L}(\psi+\varphi)\geq0$ on $(0,+\infty)\times\Omega_{+}$
if
\begin{equation}
\widetilde{Q_{1}}e^{-\gamma Z\sqrt{1+B^{2}}/B}\geq\widetilde{Q_{2}}t^{-1/2}e^{-Z^{2}/4t}e^{(-\alpha-\gamma s)t},\label{eq:___maj_Q1_Q2}
\end{equation}
for every $Z\geq1$ and $t>0$. For any such $Z$, the function
\[
g\colon t\mapsto t^{-1/2}e^{-Z^{2}/4t}e^{(-\alpha-\gamma s)t},
\]
defined on $(0,+\infty)$, attains its maximum at 
\[
t_{max}=\frac{\sqrt{1+4(\gamma s+\alpha)Z^{2}}-1}{4(\gamma s+\alpha)}>0,
\]
which, since $\sqrt{1+a^{2}}-1\leq a\leq\sqrt{1+a^{2}}$ for any $a\geq0$,
leads to 
\begin{align*}
g(t) & \leq Q_{3}\exp\left(-\frac{(\gamma s+\alpha)Z^{2}}{\sqrt{1+4(\gamma s+\alpha)Z^{2}}-1}\right)\exp\left(-\frac{\sqrt{1+4(\gamma s+\alpha)Z^{2}}}{4}\right)\\
 & \leq Q_{3}e^{-\sqrt{\gamma s+\alpha}Z},
\end{align*}
with $Q_{3}\coloneqq t_{max}^{-1/2}e^{1/4}>0$. Therefore, \eqref{eq:___maj_Q1_Q2}
amounts to
\[
\widetilde{Q_{1}}e^{-\gamma Z\sqrt{1+1/B^{2}}}\geq\widetilde{Q_{2}}Q_{3}e^{-\sqrt{\gamma s+\alpha}Z}.
\]
Now, let us recall that since $\alpha\geq\gamma^{2}\left(1+\frac{1}{B^{2}}\right)$,
we have
\[
\gamma\sqrt{1+\frac{1}{B^{2}}}\leq\sqrt{\gamma s+\alpha}.
\]
Finally, by increasing $C>0$, if necessary, one has $\widetilde{Q_{1}}\geq\widetilde{Q_{2}}Q_{3}$,
thus $\mathcal{L}(\psi+\varphi)\geq0$ on $(0,+\infty)\times\Omega_{+}$. 

Putting all together, we have thus proved $\mathcal{L}(\psi+\varphi)\geq0$
on $(0,+\infty)\times\mathbb{R}^{2}$ and, from the comparison principle,
$n(t,x,y)\leq\psi(t,x,y)+\varphi(t,x,y)$ on $(0,+\infty)\times\mathbb{R}^{2}$.
In other words, for any $s>s^{\ast}$ and $\alpha\geq\gamma^{2}\left(1+\frac{1}{B^{2}}\right)$,
there exists $C>0$ such that for any $t>0$ and $(x,y)\in\mathbb{R}^{2}$,
there holds
\begin{align}
n(t,x,y) & \leq Ce^{-\gamma(X-st)}\Gamma_{0}(Y)+e^{-\alpha t}u(t,x)p(t,y).\label{eq:___n_supsol}
\end{align}
\\

We are now in the position to complete the proof of \eqref{eq:HTx_vit_c*}.
This requires an additional control of the tails of $n$ which is
postponed to Lemma \ref{lem:HTx_tails}. In the sequel, we select
$C^{\prime}>0$ and $\kappa>0$ such that \eqref{eq:HTx_n_tails}
holds. Choose any $c>c^{\ast}=s^{\ast}/\sqrt{1+B^{2}}$ and $s\in(s^{\ast},c\sqrt{1+B^{2}})$.
Now, fix $\mu>0$ and select $\zeta>0$ large enough so that
\[
C^{\prime}\int_{-\infty}^{-\zeta}e^{-\kappa|y|}dy<\frac{\mu}{2}.
\]
Additionally, there exists $T_{\alpha,\mu}\geq0$ such that for any
$t\geq T_{\alpha,\mu}$ and for all $x\in\mathbb{R}$, there holds
\begin{align*}
\int_{\mathbb{R}}e^{-\alpha t}u(t,x)p(t,y)dy & \leq e^{-\alpha t}||u_{0}||_{\infty}(\sigma_{+}-\sigma_{-})\leq\frac{\mu}{2}.
\end{align*}
Then, set $\xi\coloneqq-\zeta+Bct$. For any $t\geq T_{\alpha,\mu}$,
combining \eqref{eq:___n_supsol} and \eqref{eq:HTx_n_tails} one
obtains
\begin{align*}
\int_{\mathbb{R}}n(t,ct,y)dy & \leq\int_{-\infty}^{\xi}\left[C^{\prime}e^{-\kappa|y-Bct|}+e^{-\alpha t}u(t,ct)p(t,y)\right]dy+\int_{\xi}^{+\infty}\left[\varphi(t,ct,y)+\psi(t,ct,y)\right]dy\\
 & \leq\int_{-\infty}^{-\zeta}C^{\prime}e^{-\kappa|y|}dy+\int_{\mathbb{R}}e^{-\alpha t}u(t,ct)p(t,y)dy+\int_{\xi}^{+\infty}\psi(t,ct,y)dy\\
 & \leq\mu+C\int_{\xi}^{+\infty}\exp\left(-\gamma\left(\frac{ct+By}{\sqrt{1+B^{2}}}-st\right)\right)\Gamma_{0}\left(\frac{y-Bct}{\sqrt{1+B^{2}}}\right)dy\\
 & \leq\mu+C\sqrt{1+B^{2}}\int_{-\zeta/\sqrt{1+B^{2}}}^{+\infty}e^{-\gamma\left(\sqrt{1+B^{2}}ct+Bz-st\right)}\Gamma_{0}(z)dz\\
 & \leq\mu+\widetilde{C}e^{-\gamma\left(\sqrt{1+B^{2}}c-s\right)t}
\end{align*}
where $\widetilde{C}=C\sqrt{1+B^{2}}e^{\gamma B\zeta/\sqrt{1+B^{2}}}\int_{\mathbb{R}}\Gamma_{0}(z)dz>0$.
From there we deduce $\limsup_{t\rightarrow+\infty}\int_{\mathbb{R}}n(t,ct,y)dy\leq\mu$.
This proves \eqref{eq:HTx_vit_c*} since $\mu$ may be taken arbitrarily
small.
\end{proof}
To conclude the above proof, we require the control \eqref{eq:HTx_n_tails}. 
\begin{lem}
\label{lem:HTx_tails}Suppose $r,K$ and $n_{0}$ satisfy the assumptions
of Theorem \ref{thm:gw_HTx}. Let $n$ be any global nonnegative solution
of \eqref{eq:dim_2_xy}. Then there exists $N_{\infty}>0$ such that

\begin{equation}
\int_{\mathbb{R}}n(t,x,y)dy\leq N_{\infty}.\label{eq:n_borne_sup}
\end{equation}
Additionally, for every $\alpha>0$ there exist $C^{\prime},\kappa>0$
such that

\begin{equation}
n(t,x,y)\leq C^{\prime}e^{-\kappa|y-Bx|}+e^{-\alpha t}u(t,x)p(t,y),\label{eq:HTx_n_tails}
\end{equation}
where $u$ and $p$ respectively solve \eqref{eq:___HTx_u_eq}-\eqref{eq:___HTx_p_eq}.
\end{lem}

\begin{proof}
[Proof of Lemma \ref{lem:HTx_tails}]The proof, very similar to
that of \cite[Lemma 2.4]{AlfBerRao_17}, is included here for the
sake of completeness. The first assertion is straightforward. If we
define the mass $N(t,x)\coloneqq\int_{\mathbb{R}}n(t,x,y)dy$, an
integration of \eqref{eq:dim_2_xy} along the $y$ variable provides
the inequality
\[
\partial_{t}N-\partial_{xx}N\leq N(r_{max}-k_{-}N).
\]
Since $N_{0}(x)\leq||u_{0}||_{\infty}(\sigma_{+}-\sigma_{-})$, it
follows from the maximum principle that the mass is uniformly bounded~:
\[
N(t,x)\leq N_{\infty}\coloneqq\max\left(||u_{0}||_{\infty}(\sigma_{+}-\sigma_{-}),\frac{r_{max}}{k_{-}}\right),
\]
which proves \eqref{eq:n_borne_sup}.

Let us now turn to the second assertion. Fix $R>0$ large enough such
that $\tilde{r}(Y)\leq-\alpha$ whenever $|Y|\geq R$. Set 
\[
\Omega_{R}\coloneqq\left\{ (x,y)\in\mathbb{R}^{2}\mid|Y|=\frac{|y-Bx|}{\sqrt{1+B^{2}}}<R\right\} .
\]
Let us prove that $n$ is uniformly bounded on $\Omega_{R}$. In view
of Assumption \ref{assu:func_r_k_n0}, there exists $M>0$ such that
for all $(x,y)\in\Omega_{R+1}$, there holds
\[
\left|r(y-Bx)-\int_{\mathbb{R}}K(t,x,y,y^{\prime})n(t,x,y^{\prime})dy^{\prime}\right|\leq \Vert r \Vert_{L^\infty(\Omega_R)}+k_{+}N_{\infty}\eqqcolon M.
\]
As a consequence, $n$ is the solution of a linear parabolic problem
with bounded coefficients on $\overline{\Omega_{R+1}}$ (the nonlocal
term being treated as a function of $(t,x,y)$), which allows us to
apply the parabolic Harnack inequality (see \cite{Mos_64} for instance).
Fix any $\tau>0$. There exists $C_{H}=C_{H}(\tau,R)>0$ such that
for all $t>0$ and $\overline{x}\in\mathbb{R}$~: 
\begin{align*}
\max_{(x,y)\in B_{R}(\overline{x})}n(t,x,y) & \leq C_{H}\min_{(x,y)\in B_{R}(\overline{x})}n(t+\tau,x,y),
\end{align*}
where $B_{R}(\overline{x})\subset\Omega_{R+1}$ denotes the closed
ball of radius $R$ of center $(\overline{x},B\overline{x})$. This
yields
\[
\max_{(x,y)\in B_{R}(\overline{x})}n(t,x,y)\leq\frac{C_{H}}{2R}\int_{\mathbb{R}}n(t+\tau,\overline{x},y)dy=\frac{C_{H}}{2R}N(t+\tau,\overline{x})\leq\frac{C_{H}N_{\infty}}{2R}.
\]
Seeing that $C_{H}$ does not depend on $\overline{x}$, the population
$n(t,x,y)$ is uniformly bounded by $\frac{C_{H}N_{\infty}}{2R}$
on $\mathbb{R}_{+}\times\overline{\Omega_{R}}$.

To conclude, define, for any $\alpha>0$,
\[
\varphi(t,x,y)\coloneqq Ce^{-\kappa(|y-Bx|-R\sqrt{1+B^{2}})}+e^{-\alpha t}u(t,x)p(t,y),
\]
where $C,\kappa$ are positive constants, and $u,p$ solve \eqref{eq:___HTx_u_eq}
and \eqref{eq:___HTx_p_eq} respectively. Let us check that $n(t,x,y)\leq\varphi(t,x,y)$
on $\mathbb{R}_{+}\times\mathbb{R}^{2}$. The inequality holds for
$t=0$ by \eqref{eq:borne_n0_HTx}. If we choose any $C\geq\frac{C_{H}N_{\infty}}{2R}$,
there holds $n(t,x,y)\leq\varphi(t,x,y)$ on $\mathbb{R}_{+}\times\overline{\Omega_{R}}$.
Next, on the remaining region $(0,+\infty)\times\Omega_{R}^{c}$,
we have $r(y-Bx)\leq-\alpha$, thus $\varphi$ satisfies

\begin{align*}
\partial_{t}\varphi-\partial_{xx}\varphi-\partial_{yy}\varphi-r(y-Bx)\varphi= & \left(-\kappa^{2}(1+B^{2})-r(y-Bx)\right)Ce^{-\kappa(|y-Bx|-R\sqrt{1+B^{2}})}\\
 & +(-\alpha+||u_{0}||_{\infty}-u(t,x)-r(y-Bx))e^{-\alpha t}u(t,x)p(t,y),\\
\geq & \left(\alpha-\kappa^{2}(1+B^{2})\right)Ce^{\kappa R\sqrt{1+B^{2}}}e^{-\kappa|Y|},
\end{align*}
which is nonnegative if we fix any $\kappa\leq\sqrt{\frac{\alpha}{1+B^{2}}}$.
Since $n\geq0$, it is a subsolution of the same operator. The maximum
principle allows us to conclude that $n(t,x,y)\leq\varphi(t,x,y)$
on $(0,+\infty)\times\Omega_{R}^{c}$. We deduce \eqref{eq:HTx_n_tails}
by setting $C^{\prime}\coloneqq Ce^{\kappa R\sqrt{1+B^{2}}}$.
\end{proof}

\paragraph{Acknowledgements.}

This research was supported by the ANR I-SITE MUSE, project MICHEL
170544IA (n° ANR-IDEX-0006). The author would like to thank J. Coville
for raising the issue of ill-directed heavy tails, and his advisor
M. Alfaro, for his regular support and useful remarks.

\bibliographystyle{siam}
\bibliography{biblio}

\end{document}